\documentclass[12pt]{amsart}
\usepackage{times}
\usepackage[T1]{fontenc}
\usepackage[utf8]{inputenc}
\usepackage{lmodern}
\usepackage[english]{babel}
\usepackage{geometry}
\usepackage{booktabs}

\usepackage{amsthm}
\usepackage{amsmath,amssymb}
\usepackage{emptypage}
\usepackage{xfrac}    
\usepackage{faktor}
\usepackage{graphicx}
\usepackage{listings}
\usepackage{lipsum}
\usepackage{setspace}
\usepackage{lipsum}
\usepackage{newlfont}
\usepackage[mathscr]{eucal}	
\usepackage{physics}
\usepackage{tikz}
\usepackage{url}
\usepackage{hyperref}
\usetikzlibrary{positioning}
\usepackage{xcolor}
\usepackage[new]{old-arrows}
\usepackage{multirow}

\theoremstyle{plain}
\newtheorem{theorem}{Theorem}[section]
\newtheorem{lemma}[theorem]{Lemma}

\newtheorem{corollary}[theorem]{Corollary}

\theoremstyle{definition}
\newtheorem{definition}[theorem]{Definition}

\newtheorem{example}[theorem]{Example}
\newtheorem{prop}[theorem]{Proposition}
\newtheorem{remark}{Remark}[section]

\usepackage{setspace}
\usepackage[signatures,write]{frontespizio}

\usepackage{fancyhdr}

\newcommand{\Z}{\mathbb{Z}} 
 
\newcommand{\R}{\mathbb{R}}

\AtBeginDocument
{
\setlength\abovedisplayskip{.6\baselineskip}
\setlength\abovedisplayshortskip{.6\baselineskip}
\setlength\belowdisplayskip{.5\baselineskip}
\setlength\belowdisplayshortskip{.5\baselineskip}
}
\begin{document}

\title{Generalized K\"ahler manifolds via mapping tori}

\author{Beatrice Brienza}
\address[Beatrice Brienza]{Dipartimento di Matematica ``G. Peano'', Universit\`{a} degli studi di Torino \\
Via Carlo Alberto 10\\
10123 Torino, Italy,}
\email[Beatrice Brienza]{beatrice.brienza@unito.it}

\author{Anna Fino}
\address[Anna Fino]{Dipartimento di Matematica ``G. Peano'', Universit\`{a} degli studi di Torino \\
Via Carlo Alberto 10\\
10123 Torino, Italy,
\& Department of Mathematics and Statistics, Florida International University\\
Miami, FL 33199, United States,}
\email[Anna Fino]{annamaria.fino@unito.it, afino@fiu.edu}



\begin{abstract}

Starting from the product of  a  $3$-torus and a  compact K\"ahler (respectively,  hyperK\"ahler) manifold we construct via mapping tori   generalized K\"ahler manifolds  of split (respectively,  non-split) type. 
In this way we  obtain  new non-K\"ahler  examples and we recover   the known examples of generalized K\"ahler solvmanifolds.  Moreover,   we  investigate the formality and the  Dolbeault cohomology of the generalized K\"ahler mapping tori. 
\end{abstract}

\maketitle

\section{Introduction}

Generalized Kähler structures were introduced and studied by M.  Gualtieri \cite{Gualtieri}  in the context of N. Hitchin’s generalized complex geometry \cite{Hitchin}. From the physics point of view, 
 they  are the general solution to the $(2, 2)$ supersymmetric sigma model \cite{GHR}.
 
A {\em generalized K\"ahler structure}  on a $2n$-dimensional manifold $M$ is  given by a pair of commuting complex structures $({\mathcal J}_1 ,  {\mathcal J}_2)$ on the vector bundle $T M \oplus T^* M$, 
which are integrable with respect to the (twisted) Courant bracket on $TM \oplus T^*M$, are compatible with the natural inner product  $\langle  \cdot, \cdot  \rangle$ of signature $(2n, 2n)$ on $T M \oplus T^*M$  and such that   $\langle  {\mathcal J}_1  \cdot,  \mathcal J_2 \cdot  \rangle$ is positive definite. A manifold $M$ endowed with such a pair  $({\mathcal J}_1 ,  {\mathcal J}_2)$ is called \emph{generalized K\"ahler manifold}.

By  \cite{Gualtieri, AG} a generalized K\"ahler structure on a  $2n$-dimensional  manifold  $M$  can be also described as a triple $(g,J_\pm)$ where $g$ is a Riemannian metric on $M$ and $J_\pm$ is a pair of complex structures compatible with $g$ and such that the following occurs
\[
  d^c_+\omega_+=-d^c_-\omega_-=H,
\]
where $H$ is a closed $3$-form on $M$, $\omega_\pm$ are the fundamental forms of $(g,J_\pm)$ respectively and $d^c_\pm$ are the real Dolbeault operator associated to the complex structures $J_\pm$. In particular, any K\"ahler metric $g$  on a complex manifold $(M, J)$ gives rise to a trivial generalized K\"ahler structure by taking $J_+ = J$ and  $J_- = \pm J$.
The $3$-form $H$  is also called the torsion form of the generalized K\"ahler structure  and it can be identified with the torsion of the Bismut (or Strominger) connection associated with the Hermitian structure $(J_{\pm}, g)$ (\cite{Bismut, Gaud}).  Clearly, if $H\equiv 0$, then the underlying manifold is K\"ahler.

In \cite{Hitchin} N. Hitchin  proved that if a complex manifold $(M, J)$ has a generalized Kähler structure $(J_{\pm}, g)$  such that $J = J_+$ and $J_{\pm}$ do not commute, then the commutator $[J_+, J_-]$  defines a holomorphic Poisson structure  on $(M, J)$  and in this case, the generalized K\"hler structure is called non-split. 
If the complex structures $J_+$ and $J_-$ commute, the generalized Kähler structure is said to be split since $Q = J_+ J_-$ is an involution of the tangent bundle  $T M$, which splits as a direct sum of the $(\pm 1)$-eigenspaces
 of  $Q$ (\cite{AG}).

A generalized K\"ahler structure on a smooth manifold $M$ is said to be \emph{split} if $[J_+,J_-]=0$ and \emph{non-split} otherwise. 
A split generalized K\"ahler structure such that $J_+ \neq \pm J_-$ is  also called of \emph{ambi-hermitian type}.

Complex surfaces  admitting generalized K\"ahler structures   of split type have been classified in \cite{AG}. There are many explicit constructions of non-trivial generalized K\"ahler structures, e.g., \cite{ADE,AGG, AG, BM, CG, DM, FP21, FP22, FT}. In particular,  non-K\"ahler  examples are given by   (compact) solvmanifolds, in contrast with the case of (compact) nilmanifolds which cannot admit any invariant generalized K\"ahler structures unless they are tori  \cite{Cavalcanti}.

 K\"ahler manifolds have very restrictive topological properties,  like for instance  they have even odd Betti numbers, satisfy the strong Lefschetz property and are formal in the sense  of Rational Homotopy Theory \cite{Sullivan, DGMS}.
 However very little is known in general about the differential topology of generalized K\"ahler manifolds.

Since  solvmanifolds  admit a closed non-vanishing $1$-form,   they can  be described as mapping tori   (\cite{Tischler}).
 Recall that  given a diffeomorphism  $f$ of  a  manifold $M$, the \emph{mapping torus} $M_f$ of $M$ is the quotient
\[
   M_f=\frac{M\times[0,1]}{(x,0)\sim (f(x),1)}.
\]
The manifold $M_f$  can be also viewed  as the quotient of $M \times \mathbb{R} $ with respect to the $\mathbb{Z}$ action $n\cdot(x,t)\mapsto (f^n(x),t+n)$ and  one has the natural the fibration $\pi: M_f \to S^1, \ [(x,t)] \mapsto e^{2\pi i t}$. \\
So a natural question is whether there exists a way to construct new examples of generalized K\"ahler manifolds via mapping tori.

In  the paper we construct  new generalized K\"ahler  mapping tori starting from the product  $M = {\mathbb T}^3 \times N$ of  a $3$-torus $\mathbb{T}^3$,  endowed with a $1$-parameter family of normal almost contact metric structures, and  a compact  K\"ahler (respectively, hyperK\"ahler) manifold $N$. We fix on the $3$-torus a family of basis $\{e^i(t)\}$ of $1$-forms depending on a time variable $t$ running over $[0,1]$ and obeying certain conditions.  We prove that,  if $f$ is a block map $(\rho, \psi)$, where $\rho$ is a diffeomorphism of $\mathbb T^3$  preserving the basis $\{e^i(t)\}$ and $\psi$ is a holomorphic automorphism of $N$ preserving the K\"ahler (respectively, hyperK\"ahler) structure, then the mapping torus $M_f$ is endowed with a split (respectively, non-split) generalized K\"ahler structure $(g,I_\pm)$ (Theorems \ref{lemma:2} and \ref{lemma:1}). 

In Section  \ref{sectformality} we prove that   the  generalized K\"ahler manifolds   $M_f$  have odd first Betti number  (Proposition \ref{proposition:nokahler}) and some results about the formality.
In particular, we show that, if the diffeomorphism  $\psi$  of $N$  is the identity map,    then $M_f$ is simply the product of $ N \times \mathbb{T}^3_\rho $, where $\mathbb{T}^3_\rho$ is the mapping torus of $\mathbb{T}^3$ with respect to the diffeomorphism $\rho$ and  therefore  is formal (Corollary \ref{corollary:cor1}). We also give a more explicit description of $\mathbb{T}^3_\rho$, which is actually biholomorphic to an Inoue surface in the family $S_M$ (Lemma \ref{lemma:splitmt}).

In Theorem \ref{theorem:holfib} we prove  that  the generalized K\"ahler manifolds $(M_f, I_\pm, g)$  constructed in Theorems  \ref{lemma:2} and \ref{lemma:1} are the total space of a holomorphic fibration over the  Inoue surface  $\mathbb{T}^3_\rho$ with compact K\"ahler (respectively, hyperK\"ahler) fibre $N$. Therefore the Dolbeault cohomology of $(M_f, I_\pm)$  can be computed using  the  two Borel spectral sequences associated to the holomorphic fibration. \\
In the last section  we give  some explicit examples of the constructions described in Theorems \ref{lemma:2} and \ref{lemma:1}. More precisely,  considering  as   hyperK\"ahler manifold  $N$ a $4$-torus, we recover the known examples of generalized K\"ahler solvmanifolds.  Moreover, we show that examples not diffeomorphic to solvmanifolds can be constructed by taking  as the (hyper)K\"ahler manifold $N$  a $K3$ surface.

\section{Formality of mapping torus}

We recall some basic facts of the theory of minimal models and formality  (\cite{DGMS}).

\begin{definition}
A \emph{commutative differential graded algebra}  (CDGA for short) is a graded algebra $A=\oplus_k A^k$ which is graded commutative, together with a differential $d: A^k \to A^{k+1}$ satisfying the Liebinitz rule and such that $d^2=0.$\\
A morphism of CDGA is a degree-preserving linear map, which preserves the multiplications of the underlying algebras and commutes with the differentials.
\end{definition}
Observe that from a CDGA $(A,d)$ one can always construct its cohomology algebra $H^*(A,d)$ which can be turned into a CDGA once endowed with the zero differential.\\
A CDGA morphism is said to be a \emph{quasi isomorphism} if the induced map in cohomology is an isomorphism. \\
A CDGA $(A,d)$ such that $H^0(A,d)=\mathbb{R}$ is said connected. \\
The basic example of CDGAs is the complex of differential forms of a smooth manifold $M$ endowed with the exterior derivative, which we will denote by $(\Omega^*(M),d)$. \\
\begin{definition}
A CDGA $(A,d)$ is said to be \emph{minimal} if the following occurs:
\begin{itemize}
    \item $A=\bigwedge V$ is the free commutative algebra generated by a graded real vector space $V$;
    \item $V$ is endowed with an homogeneous basis $\{x^i\}_{i \in I}$, where $I$ is a well ordered set, such that $\abs{x^i}\le \abs{x^j}$ for $i<j$ and $dx^i=\bigwedge \{x^j\}_{j<i}$.
\end{itemize}
\end{definition}
By  \cite{SH}
every connected CDGA $(A,d)$ has an unique minimal model, up to isomorphism. That is, a minimal algebra $(\bigwedge V,d)$ together with a quasi isomorphism $\varphi:(\bigwedge V,d)\to (A,d)$.

Moreover, by definition, the minimal model of a connected manifold is the minimal model of the CDGA $(\Omega^*(M),d)$. Hence, the minimal model of a smooth (connected) manifold allow us to encode all the (De Rham) cohomology of the manifold with a small amount of algebraic relations.
\begin{definition}
A minimal algebra $(\bigwedge V,d)$ is said to be \emph{formal} if there exists a CDGA morphism $\nu:(\bigwedge V,d)\to(H^*(\bigwedge V,d),0) $ inducing the identity in cohomology.
\end{definition}
A connected smooth manifold is said to be formal if its minimal model is formal.

\smallskip

In general, the computation of the minimal model of a mapping torus $M_f$  is not trivial, but some sufficient conditions for the formality are known.  By \cite[Lemma 12] {BFM}, the cohomology of  $M_f$  fits in an exact sequence
\[
 0 \to C^{r-1} \to H^r(M_f) \to K^r \to 0, 
\]
where $K^r=\ker(f_r^*-Id)$ and $C^{r-1}=\operatorname{coker}(f_{r-1}^*-Id)$. Moreover, since $K^r$ is free, then the exact sequence splits, i.e. 
\[
   H^r(M_f)=K^r \oplus  C^{r-1}.
\]

We recall two results about the formality of mapping tori and the construction of their minimal model. The result below is contained in \cite[Theorem 13, 15] {BFM}.
\begin{theorem} 
Let $M$ be an oriented compact smooth manifold of dimension $n$ and let $f:M \to M$ be an orientation preserving diffeomorphism. Let $M_f$ be the mapping torus of $f$. The following two results hold:
\begin{itemize}
    \item Suppose that for some $p > 0$ the isomorphism $f^*_p:H^p(M)\to H^p(M) $ has eigenvalue $\lambda=1$ with multiplicity $2$. Then $M_f$ is non-formal. \\
    \item Suppose that there is some $p \ge 2$ such that $f_k^*$ does not have eigenvalue $\lambda=1$ for any $k \le p-1 $, and that $f_p^*$ does have the eigenvalue $\lambda=1$ with some multiplicity $r \ge 1$. Denote
\[
    K^j=\ker(f_p^*-Id)^j
\]
for any $j=0,\dots,r$. So $\{0\}=K^0 \subset K^1 \subset \dots \subset K^r$. Write $G^j=K^j / K^{j-1}$, for $j=1,\dots,r$. The map $F=f_p^*-Id$ induces maps $F:G^j \to G^{j-1}$, $j=1,\dots,r$.\\
Then the minimal model of $M_f$ is, up to degree $p$, given by
\[
\begin{split}
    &V^1=\langle a \rangle, \ \ \ da=0, \\
    &V^k=0, \ \ \ k=2,\dots,p-2,\\
    &V^p=G^1 \oplus G^2\oplus \dots \oplus G^r, \ \ \ dv= a \cdot F(v), \ \ v\in G^j.
\end{split}
\]
Moreover, if $r \ge 2$, then $M_f$ is non-formal.
\end{itemize}
\end{theorem}

\section{Split generalized K\"ahler  mapping tori} \label{section3}

Complex surfaces admitting a generalized K\"ahler  structure of split type have been classified in \cite{AG}. The ones   topologically equivalent to $3$-torus bundles over $S^1$ and having odd first Betti number are biholomorphic to Inoue surfaces in the family $S_M$. 
In the next Example, we  describe  an example of Inoue surface in the family $S_M$ via mapping tori.  \\

\begin{example} \label{example:basis} 
Let $p$ and $t_0$ be a pair of non-zero real numbers such that the following matrix
\begin{equation} \label{eqn:rho}
  \rho(t_0)= \begin{pmatrix}
              e^{t_0} & 0 & 0 \\
              0 & e^{-\frac{t_0}{2}}\cos(t_0p) &  e^{-\frac{t_0}{2}}\sin(t_0p) \\
              0 & -e^{-\frac{t_0}{2}}\sin(t_0p) & e^{-\frac{t_0}{2}}\cos(t_0p) \\
              \end{pmatrix}  
\end{equation}
is similar to an integer matrix, say $A$. The existence of such a pair is proved in \cite[Section 3.2.2]{AO}. Hence, there exists an invertible matrix $P$ such that $PA=\rho(t_0)P$. A lattice $\Gamma_0$ in $\mathbb{R}^3$ defined by
 \[
  \Gamma_0=P(m_1,m_2,m_3)^t, 
 \]
where $m_1,m_2,m_3 \in \Z$ and $(m_1,m_2,m_3)^t$ is the transpose of the vector $(m_1,m_2,m_3)$, is invariant under $\rho(t_0)$. Thus $\Gamma_0$ is a cocompact subgroup of $\mathbb{R}^3$ isomorphic to $\mathbb{Z}^3$. \\
Since $\rho(t_0)$ preserves $\Gamma_0$, we may regard $\rho(t_0)$ as a diffeomorphism of  the $3$-torus $\mathbb{T}^3= \Gamma_0  \backslash \mathbb{R}^3$. From now on we  will denote $\rho(t_0)$ simply by $\rho$. \\ 
For any $t\in[0,t_0]$,  fix the following frame of $\mathbb{T}^3$:
$$
e_1(t)=e^t \frac{\partial}{\partial x^1}, \, 
e_2(t)=e^{-\frac{t}{2}} \cos(pt) \frac{\partial}{\partial x^2}-e^{-\frac{t}{2}} \sin(pt) \frac{\partial}{\partial x^3},\,
e_3(t)=e^{-\frac{t}{2}} \sin(pt) \frac{\partial}{\partial x^2}+e^{-\frac{t}{2}} \cos(pt) \frac{\partial}{\partial x^3}, 
$$
with dual coframe 
$$ 
e^1(t)=e^{-t} dx^1, \,
e^2(t)=e^{\frac{t}{2}} \cos(pt) dx^2-e^{\frac{t}{2}} \sin(pt) dx^3,\,
e^3(t)=e^{\frac{t}{2}} \sin(pt) dx^2+e^{\frac{t}{2}} \cos(pt) dx^3.
$$
For any $t$ we can define the following pair of $1$-parameter family of normal almost contact metric structures $(\xi(t),  \eta(t), \phi_{\pm}(t), h(t))$,  given by 
\[
\begin{split}
    &\xi (t) :=e_1 (t), \quad \eta (t):=e^1(t), \quad h(t) :=\sum_{i=1}^3 (e^i (t))^{2},\\
    &\phi_\pm (t) (e_1 (t)):=0, \quad \phi_\pm (t) (e_2 (t)):=\pm e_3 (t), \quad \phi_\pm (t) (e_2 (t)):=\mp e_3 (t),
\end{split}
\]
with fundamental forms $$F_\pm(t)=\pm e^2 (t) \wedge e^3 (t) =\pm e^t \ dx^2 \wedge dx^3.$$
Note that 
\[
\begin{split}
& \rho^*\big(e^{1}(t_0)_{\rho(\underline{x})}\big)= \rho^*\big(e^{-t_0} dx^1\big)=dx^1=e^1(0)_{\underline{x}},\\
& \rho^*\big(e^{2}(t_0)_ {\rho(\underline{x}}\big)=\rho^*\big(e^{\frac{t_0}{2}} \cos(pt_0) dx^2-e^{\frac{t_0}{2}} \sin(pt_0) dx^3 \big)= dx^2 =e^2(0)_{\underline{x}}, \\
& \rho^*\big(e^{3}(t_0)_ {\rho(\underline{x}}\big)=\rho^*\big(e^{\frac{t_0}{2}} \sin(pt_0) dx^2+e^{\frac{t_0}{2}} \cos(pt_0) dx^3 \big)= dx^3 =e^3(0)_{\underline{x}}.
\end{split}
\]
From $\rho^*\big(e^{i}(t_0)_{\rho(\underline{x})}\big)= e^i(0)_{\underline{x}}$, we also have
\[
 \rho_{*}\big( e_i(0)_{\underline{x}}\big)= e_i(t_0)_{\rho(\underline{x})}, \quad   i=1,2,3. 
 \]
If we regard $e^i$ as a $1$-form on $\mathbb{T}^3 \times [0,t_0]$, then the computation above shows that $e^i$ is preserved by the smooth map 
$\rho: \mathbb{T}^3 \times \{0\} \to \mathbb{T}^3 \times \{t_0\}$ sending $(\underline{x},0) \mapsto (\rho(\underline{x}),t_0)$, for each $i=1,2,3$. Hence $\{e^i\}$ descends to a coframe  on the mapping torus 
\[
   \mathbb{T}^3_\rho= \frac{\mathbb{T}^3 \times [0,t_0]}{(\underline{x},0) \sim (\rho(\underline{x}),t_0) }.
\]
For the same reason,   $\{ e_i \} $ induces  a basis of  vector fields on the mapping torus.

Let $$
\pi: \mathbb{T}^3_\rho \to S^1,    [(\underline{x},t)] \mapsto e^{2\pi i \frac{t}{t_0}}
$$
and  $\theta$ be the pullback of the standard volume form on $S^1$ via $\pi$. Note that, up to rescaling,  $\theta=dt$ locally. By previous remarks, $\{e^i, \theta\}$ and $\{e_i, \frac{\partial}{\partial \theta}\}$ are a global coframe and frame of $\mathbb{T}^3_\rho$ respectively, where $\frac{\partial}{\partial \theta}$ is the vector field 
 whose local expression is $\frac{\partial}{\partial t}$. 
 
The triple $(J_{\pm},  g)$ 
on $\mathbb{T}^3_\rho$ given by 
\[
J_\pm(e_1)=\frac{\partial}{\partial \theta}, \quad  J_\pm(e_2)=\phi_\pm (e_2), \quad  g=\sum_{i=1}^3 (e^i)^2 + \theta^2,
\]
with fundamental forms $\omega_\pm= e^1 \wedge \theta +F_\pm$, defines  a generalized K\"ahler structure of split type.

Indeed,  $N_{J_{\pm}}(e_2,e_3)=0$,  
$$
\begin{array}{lcl}
N_{J_\pm}(e_1,e_2)&=& \big[ \frac{\partial}{\partial t}, J_\pm e_2\big]-J_\pm\big(\big[ \frac{\partial}{\partial t}, e_2\big]\big) \\[4pt]
&=& \pm \big[ \frac{\partial}{\partial t}, e_3\big]-J_\pm\big(\big[ \frac{\partial}{\partial t}, e_2\big]\big)\\[4pt]
&=& \pm (p e_2 -\frac{1}{2} e_3)-J_\pm (-\frac{1}{2} e_2-p e_3)=0
\end{array}
$$
and similarly 
$$
\begin{array}{lcl}
N_{J_\pm}(e_1,e_3)&=&=\big[ \frac{\partial}{\partial t}, J_\pm e_3\big]-J_\pm\big(\big[ \frac{\partial}{\partial t}, e_3\big]\big)\\
&=&\mp \big[ \frac{\partial}{\partial t}, e_2\big]-J_\pm\big(\big[ \frac{\partial}{\partial t}, e_3\big]\big)\\
&=& \mp \big[ \frac{\partial}{\partial t}, e_2\big] \pm \big[ \frac{\partial}{\partial t}, e_2 \big]=0.
\end{array}
$$
Moreover
$$
  d^c_\pm \omega_\pm =J_\pm^{-1} d\omega_\pm
 =J_\pm^{-1} dF_\pm =\pm J_\pm^{-1} (\theta \wedge e^2 \wedge e^3) =\\
 =\mp e^1 \wedge e^2 \wedge e^3= \mp dx^1 \wedge dx^2 \wedge dx^3
$$
 is a closed $3$-form on $\mathbb{T}^3_\rho$ and 
 $[J_+,J_-]=0$. 
The first Betti number $b_1$ of $\mathbb{T}^3_\rho$ is given by 
\[
 b_1(\mathbb{T}^3_\rho)= \dim(\ker(\rho_1^*-Id))+\dim(\operatorname{coker}(\rho_0^*-Id))=\dim(\ker(\rho_1^*-Id))+1,
\]
where $\rho_1^*$ is the isomorphism induced by $\rho$ on $H^1(\mathbb{T}^3)$. With respect to  the standard basis $\{[dx^1],[dx^2],[dx^3]\}$ of $H^1(\mathbb{T}^3)$  we have 
\[
 \rho_1^*-Id=
 \begin{pmatrix}
     e^{t_0}-1 & 0 & 0 \\
     0 & e^{\frac{-t_0}{2}}\cos(t_0p)-1 & e^{\frac{-t_0}{2}}\sin(t_0p) \\
     0 & -e^{\frac{-t_0}{2}}\sin(t_0p) & e^{\frac{-t_0}{2}}\cos(t_0p)-1 \\
 \end{pmatrix}
\]
and so for any values of $p$ in $\mathbb{R}\setminus \{0\}$ $\det(\rho_1^*-Id)=(e^{t_0}-1)(e^{-t_0}-2e^{\frac{-t_0}{2}}\cos(t_0p)+1)$ vanishes  only if $t_0=0$, which is an excluded value by hypothesis. It follows that   $b_1(\mathbb{T}^3_\rho)=1$ and as a  consequence  $\mathbb{T}^3_\rho$ is an Inoue surface in the family $S_M$. 

\end{example} 

\begin{remark} \label{remark:solvmanifold} The mapping torus  $\mathbb{T}^3_\rho$  can be also described  as  the  compact  almost abelian  solvmanifold $\Gamma \backslash (\mathbb{R}^3 \rtimes_\varphi \mathbb{R})$, where  $\varphi (t)$ is given by 
\begin{equation*}
   \varphi(t)= \begin{pmatrix} 
    e^t & 0 & 0  \\
    0 & e^{-\frac{t}{2}} \cos(tp)& e^{-\frac{t}{2}} \sin(tp) \\
    0 & -e^{-\frac{t}{2}} \sin(tp) & e^{-\frac{t}{2}} \cos(tp)\\
    \end{pmatrix},
\end{equation*}
and $\Gamma= \Gamma_0 \rtimes t_0 \mathbb{Z}$.  Note that   the almost abelian Lie group $\mathbb{R}^3 \rtimes_\varphi \mathbb{R}$  has  structure equations 
 \begin{equation} \label{eqn:maurer}
    \begin{split}
        & df^1=f^1 \wedge f^{4} \ , \  df^2= -\frac{1}{2} f^2 \wedge f^{4} + p f^3 \wedge f^{4n},\\
        &df^3=-p f^2\wedge f^{4} -\frac{1}{2}f^3 \wedge f^{4}, \ df^{4}=0.
    \end{split}
\end{equation} 
and  $e^1=f^1$,\ $e^2= f^2$, \  $e^3=f^3$,\ $\theta=f^4$.
\end{remark}

We can  now show that the previous example is a particular case of a more general construction. 

Let $a_1 = a_1(t), b_2= b_2 (t), b_3 = b_3 (t),$ be  real smooth functions  depending on a time-variable $t \in [0,1]$ such that,  for every fixed  $t$,   the following vector fields  on the real $3$-torus $\mathbb{T}^3$ 
\begin{equation}  \label{fvectors(t)}
\begin{split}
    &e_1 (t)=a_1(t) \frac{\partial}{\partial x^1}, \ \  \ \  e_2(t)=b_2(t) \frac{\partial}{\partial x^2}+b_3(t) \frac{\partial}{\partial x^3}, \\
    &e_3(t)=-b_3(t) \frac{\partial}{\partial x^2}+b_2(t) \frac{\partial}{\partial x^3},
\end{split}
\end{equation}
form  a basis of  vector fields, where $(x^1,x^2,x^3)$ are local coordinates on $\mathbb{T}^3$ regarded as the quotient manifold $\Z^3 \backslash \R^3$. Then the dual basis 
\begin{equation}  \label{formei(t)}
\begin{split}
    &e^1 (t)= \frac{1}{a_1(t)} dx^1, \ \ \ \ e^2(t)=\frac{b_2(t)}{l(t)}dx^2+ \frac{b_3(t)}{l(t)} dx^3, \\
    &e^3(t)=-\frac{b_3(t)}{l(t)}dx^2+ \frac{b_2(t)}{l(t)} dx^3,
\end{split}
\end{equation}
 is, for every fixed $t$,  a basis of $1$-forms on  $\mathbb{T}^3$, where   $l(t) = b_2^{2}(t)+b_3^{2}(t)  > 0$.\\
 If we define 
 \begin{equation} \label{eqn:v}
 v(t):=b_2 (t)b_2'(t)+b_3(t)b_3'(t) = \frac 12 l'(t)
 \end{equation}
 and 
 \begin{equation*}
 w(t):=b_2 (t)b_3'(t)-b_3(t)b_2'(t),
 \end{equation*}
 we also require that 
  \begin{align} 
     & \frac{v(1)}{l(1)} =  \frac{v(0)}{l(0)}, \label{eqn:1} \\
     &  \frac{w(1)}{l(1)} =  \frac{w(0)}{l(0)}. \label{eqn:2}
 \end{align}
 Moreover,  we can construct on $\mathbb{T}^3$  the following pair of  $1$-parameter  family of normal almost contact metric structures $(\xi(t),  \eta(t), \phi_{\pm}(t), h(t))$ given by 
\[
\begin{split}
    &\xi (t) :=e_1 (t), \quad \eta (t):=e^1(t), \quad h(t) :=\sum_{i=1}^3 (e^i (t))^{2},\\
    &\phi_\pm (t) (e_1 (t)):=0, \quad \phi_\pm (t) (e_2 (t)):=\pm e_3 (t), \quad \phi_\pm (t) (e_2 (t)):=\mp e_3 (t),
\end{split}
\]
with fundamental forms $$F_\pm(t)=\pm e^2 (t) \wedge e^3 (t) =\pm \frac{1}{l(t)} \ dx^2 \wedge dx^3.$$
Starting from the basis $\{e_i\}$  given in \eqref{fvectors(t)} with coefficients $a_1,b_2,b_3$ satisfying \eqref{eqn:1} and  \eqref{eqn:2}, we  can  construct a split-generalized K\"ahler mapping torus of the $3$-torus, which is either  a K\"ahler surface  or an Inoue surface in the family $S_M$, depending on the function $v$ described above is actually zero or not. In this paper we are interested in the second case.
\begin{lemma} \label{lemma:splitmt}
Let $\{e^i (t)\}$ be the basis of  $1$-forms  on $\mathbb T^3$  as in \eqref{formei(t)} with coefficients $a_1,b_2,b_3$ satisfying \eqref{eqn:1} and  \eqref{eqn:2}.
Suppose that there exist  a diffeomorphism $\rho$  of  $\mathbb{T}^3$  such that 
\begin{equation}  \label{eqn:torus}
\begin{array}{l}
    \rho^* \big(e^{i} (1)_{\rho(\underline{x})}\big)= e^{i}(0)_{\underline{x}},  \\[5pt] 
     \big(\frac{1}{l(t)}\big)' \cdot \frac{1}{a_1}= const \neq 0.
\end{array}
\end{equation}
Then  the mapping torus $\mathbb{T}^3_\rho$ 
admits a split generalized K\"ahler structure and it is biholomorphic to an Inoue surface in the family $S_M$.
\end{lemma}
\begin{proof}
A  geometric structure on $\mathbb{T}^3 \times [0,1]$ descends on the mapping torus  $\mathbb{T}^3_\rho$  if  it is preserved by the map sending $(\underline{x},0)$ to $(\rho(\underline{x}),1)$. \\
 Therefore, by  the first condition of \eqref{eqn:torus},     $\{e^i(t)\}$ is a set of well defined $1$-forms on $\mathbb{T}^3_\rho$.  Furthermore, $F_\pm (t)$  and $\{e_i(t)\}$ both descend to  $\mathbb{T}^3_\rho$, since the pullback $\rho^*$ distributes over the wedge product and one can easily check that $\rho_*\big(e_{i}(0)_{\underline{x}}\big)=e_{i}(1)_{\rho(\underline{x})}$. From now on we  will denote by $e^i$ and $e_i$ the corresponding global  $1$-forms and vector fields on $\mathbb{T}^3_\rho$,  respectively. 
 
Let  $\theta$ be  the pull-back of the standard volume form of $S^1$ via the fibration $$\pi:\mathbb{T}^3_\rho \to S^1, \ [(\underline{x},t)]\mapsto e^{2\pi i t}$$ which, up to rescaling, locally corresponds to $dt$.  Hence $d (e^j \wedge \theta) =0$,  for every $j =1,2,3$, since  the coefficients of $dx^i$ in the expressions of $e^j$ only depends on the time-coordinate.\\
To prove the Lemma we have to construct  a pair of complex structures   $J_\pm$   on $\mathbb{T}^3_\rho$ and  a Riemannian metric  $g$ compatible with both $J_\pm$ such that $[J_+,J_-]=0$, $d^c_{-}\omega_- = -d^c_{+}\omega_+$  and $ dd^c_{+}\omega_+=0$, where we denote by $\omega_\pm$ the fundamental forms of the pairs $(g,J_\pm)$.\\
Note that the following non-degenerate $2$-forms 
\[
\omega_+=e^1\wedge \theta +F_+, \quad  \omega_-=e^1\wedge \theta+F_- .
\]
are well defined on $\mathbb{T}^3_\rho$ and 
\[
  d\omega_\pm= d(e^1 \wedge \theta)+ d F_\pm  =  d F_\pm
\]
where the last equality holds since $d(e^1\wedge \theta)=0$. 
The $3$-forms $dF_\pm$ are given by
\[
  dF_\pm= \pm d\bigg(\frac{1}{l(t)} dx^2 \wedge dx^3\bigg)=\mp \  \frac{l'(t)}{l(t)} \  dt \wedge e^2 \wedge e^3=\mp \  \frac{2v(t)}{l(t)} \ dt \wedge e^2 \wedge e^3,
\]
where $2\frac{v(t)}{l(t)}$ is the smooth function on $\mathbb{T}^3_\rho$ mapping $[(\underline{x},p,t)]$ to $ 2\frac{v(t)}{l(t)}$. 
The previous map is well defined on $\mathbb{T}^3_\rho$, by condition \eqref{eqn:1}. In terms of global $1$-forms, $dF_\pm= \mp \  \frac{2v}{l} \ \theta \wedge e^2 \wedge e^3$. \\ %
The $2$-forms $\omega_{\pm}$  are of type $(1,1)$  with respect to the almost complex structures  $J_{\pm}$, defined by
\begin{equation*}
    \begin{array}{l}
    J_\pm(e_1)= \frac{\partial}{\partial \theta}, \\[5pt]
    J_\pm(e_2)= \phi_\pm(e_2)=\pm e_3, \\ [5pt]
    J_\pm(e_3)= \phi_\pm(e_3)= \mp e_2, \\ [5pt]
    \end{array}
\end{equation*} 
where $\frac{\partial}{\partial \theta}$ is the $S^1$-vector field, whose local expression is $\frac{\partial}{\partial t}$.\\
Moreover,   $-\omega_+ J_+ = -\omega_- J_-=g$,  where 
\[
  g=\sum_{i=1}^3 (e^i)^{2}+\theta^2.
\] 
Indeed,
\[
\begin{array}{l}
     J_\pm \omega_\pm\big( e_1, \frac{\partial}{\partial \theta}\big)=\omega_\pm \big(\frac{\partial}{\partial \theta}, -e_1,\big)=\omega_\pm\big( e_1, \frac{\partial}{\partial \theta}\big),  \\ [5pt]
     J_\pm \omega_\pm(e_2,e_3)=\omega_\pm (\pm e_3, \mp e_2)=\omega_\pm (e_2,e_3), \\ [5pt]
\end{array}
\]
and $\omega_\pm= J_\pm \omega_\pm=0$ otherwise. \\
We claim the almost complex structures $J_\pm$ to be integrable by verifying the vanishing of the Nijenhuis tensor $N_{J_\pm} (X, Y) =0$, for every vector field $X$ and $Y$.
The only non-trivial checks are for 
\begin{align*}
    & X=e_1, \ \ Y\in\{e_2,e_3\}, \\ 
    & X= \frac{\partial}{\partial \theta}, \ \ Y=\{e_2,e_3\}
\end{align*}
and it is sufficient to check the vanishing of $N_{J_\pm}(X,Y)$  in a local trivialization.  Hence in the sequel we will use the local expression of $\frac{\partial}{\partial \theta}$ (which we recall to be $ \frac{\partial}{\partial t}$). \\
Let $X=e_1$ and $Y=e_2$. Then
$$
\begin{array}{lcl}
N_{J_\pm}(e_1,e_2)&=&[J_\pm e_1, J_\pm e_2]-J_\pm([J_\pm e_1, e_2]+[e_1, J_\pm e_2])-[e_1, \
e_2]\\[4pt]
&=& \big[ \frac{\partial}{\partial t}, J_\pm e_2\big]-J_\pm\big(\big[ \frac{\partial}{\partial t}, e_2\big]\big) \\[4pt]
&=& \pm \big[ \frac{\partial}{\partial t}, e_3\big]-J_\pm\big(\big[ \frac{\partial}{\partial t}, e_2\big]\big).
\end{array}
$$
Since
\[
\left[ \frac{\partial}{\partial t}, e_2\right]=b'_2(t) \frac{\partial}{\partial x^2}+b'_3(t) \frac{\partial}{\partial x^3}=\frac{v(t)}{l(t)} e_2 + \frac{w(t)}{l(t)} e_3
\]
and 
\begin{align*}
\left[ \frac{\partial}{\partial t}, e_3\right]=-b_3'(t) \frac{\partial}{\partial x^2}+b_2'(t) \frac{\partial}{\partial x^3}=\frac{-w(t)}{l(t)}e_2+ \frac{v(t)}{l(t)}e_3,
\end{align*}
we obtain
\begin{align*}
N_{J_\pm}(e_1,e_2)&=\pm \bigg( \frac{-w(t)}{l(t)}e_2+ \frac{v(t)}{l(t)}e_3 \bigg) - J_\pm \bigg( \frac{v(t)}{l(t)} e_2 + \frac{w(t)}{l(t)} e_3 \bigg)=0.
\end{align*}
Observe that all the previous objects are actually well defined on $\mathbb{T}^3_\rho$ by conditions \eqref{eqn:1} and \eqref{eqn:2}. Similarly, $N_{J_\pm}(e_1,e_3) =0$.
If $X=\frac{\partial}{\partial \theta}$ and $Y=e_i$, with $i=2,3$, we have
\begin{align*}
N_{J_\pm}\left(\frac{\partial}{\partial t},e_i\right)&=\left[J_\pm \frac{\partial}{\partial t}, J_\pm e_i\right]-J_\pm \left( \left[J_\pm \frac{\partial}{\partial t}, e_i\right]+\left[\frac{\partial}{\partial t}, J_\pm e_i\right]\right)-\left[\frac{\partial}{\partial t}, e_i\right]\\
&=-[e_1, J_\pm e_i]-J_\pm\left(-[e_1, e_i]+\left[\frac{\partial}{\partial t}, J_\pm e_i\right]\right)-\left[\frac{\partial}{\partial t}, e_i\right]\\
&=-J_\pm\left(\left[\frac{\partial}{\partial t}, J_\pm e_i\right]\right)-\left[\frac{\partial}{\partial t}, e_i\right]\\
&=\left[ \frac{\partial}{\partial t}, e_i \right]-\left[ \frac{\partial}{\partial t}, e_i \right]=0. 
\end{align*}
It remains to verify the condition on the real Dolbeault operators of $\omega_\pm$. Since $d \omega_\pm= dF_\pm= \mp \frac{2v}{l} \ \theta \wedge e^2 \wedge e^3$ we get  %
\[
\begin{split}
d^c_{+} \omega_+&= J_+^{-1} d\omega_+= J_+^{-1} dF_+= - \frac{2v}{l} J_+^{-1}(\theta \wedge e^2 \wedge e^3)=\\
&= - \frac{2v}{l} J_+^{-1} \theta \wedge J_+^{-1} e^2 \wedge J_+^{-1} e^3= \frac{2v}{l} J_+  \theta \wedge J_+ e^2 \wedge J_+ e^3=\\
&=\frac{2v}{l} \ e^1 \wedge e^2 \wedge e^3= - \bigg( \frac{1}{l}\bigg)' \cdot \frac{1}{a_1} dx^1 \wedge dx^2 \wedge dx^3, 
\end{split}
\]
while
\[
\begin{split}
d^c_{-} \omega_-&= J_-^{-1} d\omega_-= J_-^{-1} dF_-= + \frac{2v}{l} J_-^{-1}(\theta \wedge e^2 \wedge e^3)=\\
&= + \frac{2v}{l} J_-^{-1} \theta \wedge J_-^{-1} e^2 \wedge J_-^{-1} e^3= - \frac{2v}{l} J_-  \theta \wedge J_- e^2 \wedge J_- e^3=\\
&=-\frac{2v}{l} \ e^1 \wedge e^2 \wedge e^3= \bigg( \frac{1}{l}\bigg)' \cdot \frac{1}{a_1} dx^1 \wedge dx^2 \wedge dx^3.
\end{split}
\]
As a consequence  
\begin{equation} \label{eqn:ddc}
    d^c_+ \omega_+= -d^c_-\omega_-=\frac{2v}{l} \ e^1 \wedge e^2 \wedge e^3.
\end{equation}\\
By hypothesis, $\big( \frac{1}{l}\big)' \cdot \frac{1}{a_1}$ is constant, implying that $dd^c_\pm \omega_\pm=0$.\\
Observe that by construction $[J_+,J_-]=0$ and $J_+\neq \pm J_-$, hence the generalized K\"ahler structure is of  split  type. By hypothesis  $v$ is not identically zero and so as already observed in Example \ref{example:basis}, it suffices to show that  $b_1(\mathbb{T}^3_\rho)$ is odd, since we topologically have a $3$-torus bundle over $S^1$. \\
By  Corollary 1 in \cite{AG}, it is enough to prove   that the  de Rham class $[H]$ is non-trivial.  By contradiction let us assume that $H$ is exact.\\
As $v$ is not identically zero, there exists a $\bar{t}$ in $[0,1]$ such that $v(\bar{t})\neq 0$. For $t=\bar{t}$ we consider the  inclusion
$$
    \iota_{\bar{t}}: \mathbb{T}^3 \to \mathbb{T}^3_\rho, 
    \underline{x}  \mapsto [(\underline{x}, \bar{t} )].
$$
 Let $\alpha:= \iota_{\bar{t}}^*H \in \Omega^3(\mathbb{T}^3)$. Since $H$ is exact, $\alpha$ is exact too. We compute
\[
\alpha=\iota_{\bar{t}}^*H =\frac{2v (\bar{t})}{l(\bar{t})}  \ \ e^1 (\bar{t}) \wedge e^2 (\bar{t}) \wedge e^3(\bar{t}) = \frac{2v (\bar{t})}{l(\bar{t})} \ \ vol_{\mathbb{T}^3},
\]
where $vol_{\mathbb{T}^3}$ is the volume form $e^1 (\bar{t}) \wedge e^2 (\bar{t}) \wedge e^3(\bar{t})$ of $\mathbb{T}^3$. By Stokes Theorem
\[
0=\int_{\mathbb{T}^3} \alpha = \frac{2v (\bar{t})}{l(\bar{t})} \int_{\mathbb{T}^3} 
vol_{\mathbb{T}^3}\neq 0,
\]
which is an absurd. This concludes the proof.
\end{proof}

We can extend the previous construction  considering  the mapping tori of $\mathbb{T}^3 \times K$, where $\mathbb{T}^3$ is endowed with the same geometric structures as before and  $K$ is any compact  K\"ahler manifold.
\begin{theorem}\label{lemma:2}
Let $(N,   J,  k, \omega)$ 
be  a compact  K\"ahler manifold and let $\{e^i (t)\}$ be the $1$-forms  given in \eqref{formei(t)} satisfying \eqref{eqn:1} and \eqref{eqn:2}.
Suppose that there exists  a diffeomorphism $\rho$  of  $\mathbb{T}^3$ such that, 
\begin{equation} \label{eqn:mt}
\begin{array}{l} 
    \rho^* \big(e^{i} (1)_{\rho(\underline{x})}\big)= e^{i}(0)_{\underline{x}},  \\[5pt]
    \big(\frac{1}{l(t)}\big)' \cdot \frac{1}{a_1}= const \neq 0,
\end{array}
\end{equation}
and    a diffeomorphism  $\psi$  of $N$   preserving the K\"ahler structure,
i.e.   holomorphic with respect to  the complex structure $J$,  and such that  $\psi^*(\omega)=\omega$. Then  the mapping torus $M_f$  of $M=\mathbb{T}^3 \times N$ by the  diffeomorphism   
$$
f: (\underline{x},p)   \to (\rho(\underline{x}),\psi(p)), 
$$
admits a  split generalized K\"ahler structure $(I_{\pm}, g, \omega_{\pm})$. 
\end{theorem}
\begin{proof}
Since the map $f$ is a diffeomorphism of $M= \mathbb{T}^3 \times N$, we may define the mapping torus $M_f$ as the quotient 
\[
 M_f = \frac{ \mathbb{T}^3 \times N \times [0,1]}{(\underline{x},p,0) \sim (f(\underline{x},p),1)}.
\]
We have already observed that a geometric structure on $M \times [0,1]$ descends on the mapping torus if it glues up correctly in the quotient, i.e. if it is preserved by the map sending $(\underline{x},p,0)$ to $(\rho(\underline{x}),\psi(p),1)$. \\
Moreover $\{e^i(t)\}$ and $\omega$ are well defined on $M_f$ as $f$ has a block structure. Indeed, by hypothesis conditions \eqref{eqn:mt} are satisfied and $\psi$ preserves the K\"ahler structure. As in the proof of Lemma \ref{lemma:2}, $F_\pm (t)$  and $\{e_i(t)\}$ descend to $M_f$ too.\\ 
From now on we denote by $e^i$ and $e_i$ the corresponding global  $1$-forms and vector fields on $M_f$ respectively. \\
Again, we denote by $\theta$ the pull-back of the standard volume form of $S^1$. With the same argument as before, $d(e^j \wedge \theta)=0$.\\ 
To prove the Theorem we have to construct  a pair of complex structures   $I_\pm$   on $M_f$ and  a Riemannian metric $g$ compatible with both $I_\pm$ such that $d^c_{-}\omega_- = -d^c_{+}\omega_+$  and $ dd^c_{+}\omega_+=0$, where we denote by $\omega_\pm$ the fundamental forms of the pairs $(g,I_\pm)$.\\
Note that, by previous remarks, the following non-degenerate $2$-forms 
\[
\omega_+=e^1\wedge \theta +F_+ +\omega, \quad  \omega_-=e^1\wedge \theta+F_- +\omega.
\]
are well defined on $M_f$. \\
Observe that the external derivative of both $\omega_\pm$ is a $3$-form globally defined on the mapping torus. Indeed, 
\[
  d\omega_\pm= d(e^1 \wedge \theta)+ d F_\pm + d\omega=d F_\pm, 
\]
where the last equality holds since $d(e^1 \wedge \theta)=0$, and $\omega$ is closed by the K\"ahler hypothesis on $N$. If we explicit the local expression of $dF_\pm$, we get  
\[
  dF_\pm= \mp \  \frac{2v(t)}{l(t)} \ dt \wedge e^2 \wedge e^3.
\]
Observe that since $2\frac{v}{l}$ is a smooth global well defined function on $M_f$ 
by condition \eqref{eqn:1}, the global expression of $dF_\pm$ is actually $\mp \  \frac{2v}{l} \ \theta \wedge e^2 \wedge e^3$.\\
Let us denote by $D$ the smooth involutive distribution $\ker(e^1) \cap \ker( e^2) \cap \ker (e^3) \cap \ker (\theta) \subset \mathfrak{X}(M_f)$. If a vector field $X$ is in $D$, then 
\[
e^i(X)=0 \ \ \ \text{and} \ \ \  \theta(X)=0.
\] 
 Recall that a trivialization $U$ of $M_f$ is of the kind $U_1 \times W_1 \times I_1$, if $\{0,1\} \not \in I_1$ and of the kind $\pi \big(U_1 \times W_1 \times [0,\frac{1}{2}) \sqcup \rho(U_1) \times \psi(W_1) \times (\frac{1}{2}, 1]\big)$ otherwise, with 
$U_1$, $W_1$ and $I_1$ being open sets of $\mathbb{T}^3$, $N$ and $(0,1)$ respectively and $\pi: \mathbb{T}^3 \times N \times [0,1]\to M$ being the quotient map. \\
On such a trivialization we may write 
\[X_{|U}= \sum_{i=1}^3 X^i (x_i,y_j,t) \ e_i+ \sum_{j=1}^{2k} Y^j(x_i,y_j,t) \  \frac{\partial}{\partial y^j} + Z (x_i,y_j,t) \frac{\partial}{\partial t},\]
where $(x_i,y_j,t)$ are local coordinates on $U$ and $X^i, Y^j, Z$ are smooth functions on $U$.  Since $X \in D$,
\[X_{|\ U}= \sum_{j=1}^{2k} Y^j  \frac{\partial}{\partial y^j}.\]
By hypothesis, the complex structure $J$ satisfies
\[\psi_* \circ J=J \circ \psi_*,\]
and thus induces a well defined complex structure on $D$ by $C^{\infty}(M_f)$-linearity. More precisely
\[J\bigg( Y^j \frac{\partial}{\partial y^j}\bigg):= Y^j J \bigg(\frac{\partial}{\partial y^j}\bigg).\]
Although we defined $J$ only locally, the previous definition works also globally. Let us assume that $\big(\widetilde{U}, (\widetilde{x}^i,\widetilde{y}^j,\tilde{t})\big)$ is a new system of local coordinates, with no empty intersection with $U$.
Since $e^i$ and $\theta$ are global $1$-forms, 
\[
  X=\sum_{j=1}^{2k} \widetilde{Y}^j \  \frac{\partial}{\partial \widetilde{y}^j}=\sum_{j=1}^{2k} Y^j \  \frac{\partial}{\partial y^j}
\]
on $\widetilde{U} \cap U$. \\
We have to check that $J\big(\widetilde{Y}^j  \frac{\partial}{\partial \widetilde{y}^j}\big)=J \big( Y^j \frac{\partial}{\partial y^j}\big)$. We have 
\begin{align*}
    &J \bigg( Y^j \frac{\partial}{\partial y^j}\bigg)= Y^j J \bigg( \frac{\partial}{\partial y^j}\bigg)= \widetilde{Y}^k \frac{\partial y^j}{\partial \widetilde{y}^k} J\bigg(\frac{\partial \widetilde{y}^r}{\partial y^j} \frac{\partial}{\partial \widetilde{y}^r}\bigg)=\\
    &= \widetilde{Y}^k \frac{\partial y^j}{\partial \widetilde{y}^k} \frac{\partial \widetilde{y}^r}{\partial y^j} J\bigg(\frac{\partial}{\partial \widetilde{y}^r}\bigg)= \widetilde{Y}^k  J\bigg(\frac{\partial}{\partial \widetilde{y}^k}\bigg)=J\bigg( \widetilde{Y}^k  \frac{\partial}{\partial \widetilde{y}^k}\bigg),
\end{align*}
which concludes the proof of the statement above.\\
Consider  the two almost complex structures $I_\pm$ on $M_f$ be defined as follows
\begin{equation*}
    \begin{array}{l}
    I_\pm(e_1)= \frac{\partial}{\partial \theta}, \\[5pt]
    I_\pm(e_2)=\phi_\pm e_2 = \pm e_3, \\ [5pt]
    I_\pm(e_3)=\phi_\pm e_3 = \mp e_2, \\ [5pt]
    I_+=J \ \ \text{and} \ \ I_-=J \ \ \text{on $D$} \\ [5pt]
    \end{array}
\end{equation*} 
where $\frac{\partial}{\partial \theta}$ is the $S^1$-vector field, whose local expression is $\frac{\partial}{\partial t}$.\\
The $2$-forms $\omega_{\pm}$  are of type $(1,1)$  with respect $I_{\pm}$  
and satisfy $-\omega_+ I_+ = -\omega_- I_-=g$,  where 
\[
  g=\sum_{i=1}^3 (e^i)^2+k+\theta^2.
\] 
Indeed,
\[
\begin{array}{l}
     I_\pm \omega_\pm\big( e_1, \frac{\partial}{\partial \theta}\big)=\omega_\pm \big(\frac{\partial}{\partial \theta}, -e_1,\big)=\omega_\pm\big( e_1, \frac{\partial}{\partial \theta}\big),  \\ [5pt]
     I_\pm \omega_\pm(e_2,e_3)=\omega_\pm (\pm e_3, \mp e_2)=\omega_\pm (e_2,e_3), \\ [5pt]
     I_\pm \omega_\pm(X,Y)=\omega_\pm (J X , J Y)=\omega (J X , J Y)=\omega(X,Y)=\omega_\pm (X,Y), \ \ \text{if $X,Y \in D$,} \\ [3pt]
\end{array}
\]
and $\omega_\pm= I_\pm \omega_\pm=0$ otherwise. \\
We claim the almost complex structures $I_\pm$ to be integrable by verifying the vanishing of the Nijenhuis tensor
\[N_{I_\pm}=[I_\pm \cdot, I_\pm \cdot]-I_\pm([I_\pm \cdot, \cdot]+[\cdot, I_\pm \cdot])-[\cdot, \cdot]. 
\]
By the definition of $I_\pm$, if $X$ and $Y$ are in $D$, then $N_{I_{\pm}}(X,Y)=N_{J}(X,Y)=0$.\\
Observe also that the Nijenhuis tensor $N_{I_{\pm}}$ coincides with $N_{J_{\pm}}$ on $\mathfrak{X}(M_f) \setminus D$, where  $N_{J_{\pm}}$ is the Ninjenhuis tensor of the complex structures $J_\pm$ described in the proof of Lemma \ref{lemma:2}. It follows that $N_{I_{\pm_{|\mathfrak{X}(M_f) \setminus D}}}=N_{J_{\pm}}=0$.\\
It remains to check that $N_{I_\pm}(X,Y)=0$ for $X\in \big\{e_1, e_2, e_3, \frac{\partial}{\partial \theta}\big\}$ and $Y \in  D$.
As $N_{I_\pm}$ is a tensor, we can do a local computation. In the sequel we will use the local expression of $\frac{\partial}{\partial \theta}$ (which we recall to be $ \frac{\partial}{\partial t}$) and the local expression of the vector fields in $D$ described above. \\
Let $X\in \big\{e_1, e_2, e_3, \frac{\partial}{\partial \theta}\big\}$ and $Y \in  D$. We have that
\[
 N_{I_\pm}(X,Y)=N_{I_\pm}\left(X,Y^j \frac{\partial}{\partial y^j}\right)=Y^j N_{I_{\pm}} \left(X,\frac{\partial}{\partial y^j}\right)=0,
\]
concluding the proof of the integrability. \\
It remains to verify the condition on the real Dolbeault operators of $\omega_\pm$. Recall that $d^c_{\pm}= I_\pm^{-1} d I_\pm$,  but since $\omega_\pm$ are compatible with both $I_\pm$,\ $d^c_{\pm}\omega_\pm= I_\pm^{-1} d \omega_\pm.$ By previous remarks, $d \omega_\pm= dF_\pm= \mp \frac{2v}{l} \ \theta \wedge e^2 \wedge e^3.$ 
\[
\begin{split}
d^c_{+} \omega_+&= I_+^{-1} d\omega_+= I_+^{-1} dF_+= - \frac{2v}{l} I_+^{-1}(\theta \wedge e^2 \wedge e^3)=\\
&= - \frac{2v}{l} I_+^{-1} \theta \wedge I_+^{-1} e^2 \wedge I_+^{-1} e^3= \frac{2v}{l} I_+  \theta \wedge I_+ e^2 \wedge I_+ e^3=\\
&=\frac{2v}{l} \ e^1 \wedge e^2 \wedge e^3 =- \bigg( \frac{1}{l}\bigg)' \cdot \frac{1}{a_1} dx^1 \wedge dx^2 \wedge dx^3,
\end{split}
\]
while
\[
\begin{split}
d^c_{-} \omega_-&= I_-^{-1} d\omega_-= I_-^{-1} dF_-= + \frac{2v}{l} I_-^{-1}(\theta \wedge e^2 \wedge e^3)=\\
&= + \frac{2v}{l} I_-^{-1} \theta \wedge I_-^{-1} e^2 \wedge I_-^{-1} e^3= - \frac{2v}{l} I_-  \theta \wedge I_- e^2 \wedge I_- e^3=\\
&=-\frac{2v}{l} \ e^1 \wedge e^2 \wedge e^3= \bigg( \frac{1}{l}\bigg)' \cdot \frac{1}{a_1} dx^1 \wedge dx^2 \wedge dx^3.
\end{split}
\]
As a consequence  $d^c_+ \omega_+= -d^c_-\omega_-$. By hypothesis, $\big( \frac{1}{l}\big)' \cdot \frac{1}{a_1}$ is constant, implying that $dd^c_\pm \omega_\pm=0$.\\
Clearly, the generalized K\"ahler structure $(I_\pm, g, \omega_{\pm})$ constructed above is split, concluding the proof.
\end{proof}

\section{Non-split generalized K\"ahler mapping tori} \label{section4}

In this section we further generalize the construction just seen to describe examples of non-split generalized K\"ahler mapping tori. The idea is to replace the K\"ahler condition with the hyperK\"ahler one  and to use that the three complex structures  in the hypercomplex structure  satisfy $J_1 J_2=-J_2 J_1$. Hence, we will define on the integrable distribution $D$ two different complex structures which will give raise to the non-split property.
\begin{theorem}\label{lemma:1}
Let $(N,   J_1, J_2, J_3, \omega_1,\omega_2,\omega_3, k)$ be  a compact hyperK\"ahler manifold and let $\{e^i (t)\}$ be the $1$-forms  given in \eqref{formei(t)} satisfying \eqref{eqn:1} and \eqref{eqn:2}.
Suppose that there exists  a diffeomorphism $\rho$  of  $\mathbb{T}^3$  such that 
\begin{equation}  
\begin{array}{l}
        \rho^* \big(e^{i} (1)_{\rho(\underline{x})}\big)= e^{i}(0)_{\underline{x}},  \\[5pt] 
     \big(\frac{1}{l(t)}\big)' \cdot \frac{1}{a_1}= const \neq 0,
\end{array}
\end{equation}
and    a diffeomorphism  $\psi$  of $N$   preserving the hyperK\"ahler structure, 
i.e.   holomorphic with respect to every complex structure $J_i$, $i = 1,2,3$, and such that
\begin{equation} \label{eqn:hyper}
       \psi^*(\omega_i)=\omega_i.
        \end{equation}
Then  the mapping torus $M _f$ of $M= \mathbb{T}^3 \times N$  by the  diffeomorphism   
$$
f: (\underline{x},p)   \to (\rho(\underline{x}),\psi(p)), 
$$
admits a non-split generalized K\"ahler structure $(I_{\pm}, g, \omega_{\pm})$. 
\end{theorem}
\begin{proof}
The proof is similar to the previous one, so we highlight the differences between them and we pass over the details seen before. \\
We want to define a non-split generalized K\"ahler structure on the mapping torus 
\[
 M_f = \frac{ \mathbb{T}^3 \times N \times [0,1]}{(\underline{x},p,0) \sim (f(\underline{x},p),1)}.
\]
As in the previous proof, $\{e^i(t)\}$ and $(k,\omega_1,\omega_2,\omega_3)$ are well defined on $M_f$ by  \eqref{eqn:torus} and \eqref{eqn:hyper}. Furthermore, also $F_\pm (t)$  and $\{e_i(t)\}$ both descend to $M_f$, since the pullback $\rho^*$ distributes over the wedge product and one can easily check that $\rho_*\big(e_{i}(t)_{(\underline{x},0)}\big)=e_{i}(t)_{(\rho(\underline{x}),1)}$. From now on we  will denote by $e^i$ and $e_i$ the corresponding global  $1$-forms and vector fields on $M_f$, respectively. \\
Again, we denote by $\theta$ the pull-back of the standard volume form of $S^1$ via the fibration on  $S^1$ which, up to rescaling, locally corresponds to $dt$. \\
Note that since the $1$-forms $e^i$ and $\theta$ are actually the same as before, also in this case we have that $d(e^i \wedge \theta)=0$.\\
Now we have the first difference with respect to the previous construction. To define a non-split generalized K\"ahler structure, we consider  the following non-degenerate $2$-forms 
\[
\omega_+=e^1\wedge \theta +F_+ +\omega_{1}, \quad  \omega_-=e^1\wedge \theta+F_- +\omega_{2}.
\]
which are well defined on $M_f$. \\
The external derivative of both $\omega_\pm$ is a $3$-form globally defined on the mapping torus. Indeed, 
\[
  d\omega_\pm= d(e^1 \wedge \theta)+ d F_\pm + d\omega_i=d F_\pm, \ \ \text{with $i=2,3$,}
\]
where the last equality holds since $d(e^1 \wedge \theta)=0$ and $\omega_i$ are closed by the hyperK\"ahler hypothesis on $N$.  More precisely, we have 
\[
  dF_\pm= \mp \frac{2v}{l} \  \theta  \wedge e^2 \wedge e^3. 
\]
Let us denote by $D$ the smooth involutive distribution $\ker(e^1) \cap \ker( e^2) \cap \ker (e^3) \cap \ker (\theta) \subset \mathfrak{X}(M_f)$. If a vector field $X$ is in $D$, then 
\[
e^i(X)=0 \ \ \ \text{and} \ \ \  \theta(X)=0,
\] 
 and hence on a  trivialization $U$ of $M_f$  we may write 
\[X_{|\ U}= \sum_{j=1}^{4k} Y^j (x_i,y_j,t) \frac{\partial}{\partial y^j},\]
where $(x_i,y_j,t)$ are local coordinates on $U$ and $Y^j$ 
are smooth functions on $U$. \\
By hypothesis, the complex structures $J_{i}$ satisfy
\[\psi_* \circ J_i=J_i \circ \psi_*,\]
and thus induce well defined complex structures on $D$ by $C^{\infty}(M_f)$-linearity, i.e.
\[J_{i}\bigg( \sum_{j=1}^{4k}Y^j \frac{\partial}{\partial y^j}\bigg):= Y^j J_{i}\bigg(\frac{\partial}{\partial y^j}\bigg).\]
Although we defined $J_i$ only locally, the previous definition works also globally and the proof of this last statement proceeds straightforwardly as in Theorem \ref{lemma:2}. \\
Let the two almost complex structures $I_\pm$ on $M_f$ be defined as follows
\begin{equation*}
    \begin{array}{l}
    I_\pm(e_1)= \frac{\partial}{\partial \theta}, \\[5pt]
    I_\pm(e_2)=\phi_\pm e_2 = \pm e_3, \\ [5pt]
    I_\pm(e_3)=\phi_\pm e_3 = \mp e_2, \\ [5pt]
    I_+=J_{1} \ \ \text{and} \ \ I_-=J_{2} \ \ \text{on $D$} \\ [5pt]
    \end{array}
\end{equation*} 
where $\frac{\partial}{\partial \theta}$ is the $S^1$-vector field, whose local expression is $\frac{\partial}{\partial t}$.\\
The $2$-forms $\omega_{\pm}$  are of type $(1,1)$  with respect $I_{\pm}$   
and satisfy $-\omega_+ I_+ = -\omega_- I_-=g$,  where 
\[
  g=\sum_{i=1}^3 (e^i)^{2}+k+\theta^2.
 \] 
Indeed,
\[
\begin{array}{l}
     I_\pm \omega_\pm\big( e_1, \frac{\partial}{\partial \theta}\big)=\omega_\pm \big(\frac{\partial}{\partial \theta}, -e_1,\big)=\omega_\pm\big( e_1, \frac{\partial}{\partial \theta}\big),  \\ [5pt]
     I_\pm \omega_\pm(e_2,e_3)=\omega_\pm (\pm e_3, \mp e_2)=\omega_\pm (e_2,e_3), \\ [5pt]
     J_\pm \omega_\pm(X,Y)=\omega_\pm (J_i X , J_i Y)=\omega_i (J_i X , J_i Y)=\omega_i(X,Y)=\omega_\pm (X,Y), \  \ \\ [3pt]
\end{array}
\]
where in the last line $i=1,2$ respectively and $X,Y \in D$. The cases left out are trivial, as $\omega_\pm= I_\pm \omega_\pm=0$. \\
We claim the almost complex structures $I_\pm$ to be integrable by verifying the vanishing of the Nijenhuis tensor $N_{I_\pm}$.
This proof is the same as in Theorem \ref{lemma:2}, so we skip it to not be too repetitive. \\
It remains to verify the condition on the real Dolbeault operators of $\omega_\pm$. By previous remarks, $d \omega_\pm= dF_\pm= \mp \frac{2v}{l} \ \theta \wedge e^2 \wedge e^3$ and we have already computed $d^c_+ \omega_+= -d^c_-\omega_-$.\\
As before, since by hypothesis $\big( \frac{1}{l}\big)' \cdot \frac{1}{a_1}$ is constant,  $dd^c_\pm \omega_\pm=0$.\\
The generalized K\"ahler structure $(I_\pm, g, \omega_{\pm})$ constructed above is non-split. Indeed, since $I_\pm$ have both a block structures, it suffices to check that $I_+I_- \neq I_- I_+$ on a vector field $X$ in the block corresponding to $D$. Let $X$ be in $D$.
Then
\[
\begin{split}
   I_+ I_- X&= J_1 J_2 X = J_3 X \\
   I_- I_+ X&= J_2 J_1 X = -J_3 X.
\end{split}
\]
Alternatively, one may compute the Poisson tensor $\sigma=\frac{1}{2}[I_+,I_-]g^{-1}=-\omega_{3}^{-1} \neq 0$. 
\end{proof}

\begin{remark}
The previous proof adapts also in the case of $N$ being a generalized K\"ahler manifold $(N,J_\pm, \sigma_\pm,k)$ and $\psi$ being a diffeomorphism preserving the generalized K\"ahler structure, i.e. holomorphic with respect to $J_\pm$ and such that $\psi^*\sigma_\pm=\sigma_\pm$. In this latter case we set $\omega_\pm=e^1 \wedge \theta \ \pm \ dF_\pm + \ \sigma_\pm$ and $I_\pm=J_\pm$ on $D$. Since $[I_+,I_-]=[J_+,J_-]$ on $D$, $( I_\pm, g)$ will be a split generalized K\"ahler structure on $M_f$ if so is  $(J_\pm,k)$ on $N$, and non-split otherwise.  

\end{remark}

\section{Formality}  \label{sectformality}
In this section we prove  some results about the   formality  of the generalized K\"ahler manifolds constructed in  Sections \ref{section3} and \ref{section4}.

\smallskip

We   first show   the following proposition.

\begin{prop}  \label{proposition:nokahler}
Let $M_f$ be a generalized K\"ahler mapping torus constructed as in Theorems \ref{lemma:2} and \ref{lemma:1}, then  the first Betti number of  $M_f$ is odd. As a consequence, $M_f$  does not  admit any K\"ahler metric and  does  not satisfy the $dd^c$-lemma.
\end{prop}
\begin{proof}
Recall that
\[
 b_1(M_f)=1+\dim(\ker(\rho_1^*-Id))+ \dim (\ker(\psi_1^*-Id)).
\]
Since $\mathbb{T}^3_\rho$ is biholomorphic to an Inoue surface in the family $S_M$  it follows that    $\dim(\ker(\rho_1^*-Id))=0$. Hence, we reduce to show that the vector subspace $$\dim (\ker(\psi_1^*-Id))=\{[\alpha]\in H^1(N)\ | \ \psi_1^*([\alpha])=[\alpha]\}\subset H^{1}(N) $$ has even dimension. \\
Let $[\alpha]$ be in $\ker(\psi_1^*-Id)$. Without loss of generality, we may assume that $\alpha$ is the harmonic representative of its cohomology class, as $N$ is compact. As $[\alpha]$ is in  $\ker(\psi_1^*-Id)$, $\psi^*(\alpha)=\alpha + d \eta$ for some smooth function $\eta$ on $N$. Let us denote by $( \cdot, \cdot)$ the $L^2$ product defined on $(N,k)$ where $k$ is the K\"ahler (respectively, hyperK\"ahler) metric  on $N$.  We claim $d\eta$ to be zero. Indeed,
\[
\begin{split}
0&=(\psi^*(\Delta \alpha), d \eta)=(\Delta(\psi^*\alpha), d \eta)=(\Delta(\alpha+d\eta), d \eta)=\\
&=(\Delta(d\eta), d \eta)=(dd^*d\eta, d\eta)=(d^*d\eta, d^*d\eta)= \norm{d^*d\eta}^2,\\
\end{split}
\]
where the second equality holds since the pullback by an isometry commutes with the laplacian operator.\\
Since $d^*d\eta=0$, then $\Delta(d\eta)=0$. Hence, $\alpha$ and $\alpha+d\eta$ are two harmonic representatives of the same cohomology class. By uniqueness, $d\eta=0$.\\
We just proved that, if $[\alpha]$ is in the $\ker(\psi_1^*-Id)$ and $\alpha$ is the harmonic representative, then $\psi^*\alpha=\alpha$. \\
Recall that on K\"ahler manifolds, $J$ induces a map
\begin{align*}
    J: \cal{H}^1(N) &\to \cal{H}^1(N), \beta \mapsto J\beta,
\end{align*}
where $J \beta(X)=\beta (JX)$ and $\cal{H}^1(N)$ is the vector space of harmonic $1$-forms. Observe that the latter map is well defined since $J$ commutes with the laplacian operator. 
In the hyperK\"ahler case we may take $J=J_1$. \\
If $\alpha$ is the harmonic representative of $[\alpha] \in \ker(\psi^*_1-Id)$, then also $[J\alpha]$ is in $\ker(\psi^*_1-Id)$. Indeed, if we apply $\psi^*$ to $J\alpha$ we get
\[
   \psi^*(J\alpha)=J(\psi^*\alpha)=J\alpha.
\]
This concludes the proof of the first statement, since up to choose harmonic representatives, both $[\alpha]$ and $[J\alpha]$ are in $\ker(\psi_1^*-Id)$. \\
The proof of the second statement is an obvious consequence of the oddity of $b_1(M)$. By \cite{PG}, a compact complex manifold can satisfy the $dd^c$-lemma only if the first Betti number of the manifold is even. The thesis follows. 
\end{proof}

\begin{remark} An alternative proof of the Proposition \ref{proposition:nokahler} can be done by adapting in our setting the arguments used by Bazzoni, Lupton and Oprea in Proposition 2.3 and Corollary 2.4 of the work \cite{BLO}. 
\end{remark}
In order to study the formality of $M_f$  we first  prove  the following  Lemma.
\begin{lemma}
The $4$-dimensional generalized K\"ahler mapping tori  ${\mathbb T}^3_{\rho}$ constructed in Lemma \ref{lemma:splitmt} are formal.
\end{lemma}
\begin{proof}  
It follows from the fact that $(\mathbb{T}^3_\rho, \omega_\pm, J_\pm)$ is biholomorphic to an Inoue surface in the family $S_M$.  Indeed,  the  Rham cohomology of $\mathbb{T}^3_\rho$  is completely described in \cite{AS} and it is given by
\begin{equation} \label{equation:derhammt}
  H^{\bullet} (\mathbb{T}^3_\rho) = \mathbb{R}\langle 1 \rangle  \oplus \mathbb{R}\langle [\theta] \rangle \oplus \mathbb{R}\langle [e^{123}] \rangle \oplus \mathbb{R}\langle [\theta \wedge e^{123}] \rangle.  
\end{equation}
Therefore,  
\[
   \big( (\bigwedge V= \bigwedge (a) \otimes \bigwedge (b),d),\varphi\big),
\]
where $|a|=1$, \ $|b|=3$, \ $da=db=0$  and 
\[
 \varphi(a)=\theta, \ \  \varphi(b)=e^1 \wedge e^2 \wedge e^3,
\]
is the minimal model of $S_M$. If we define $\nu: (\bigwedge V,d) \to (H^{\bullet} (\bigwedge V,d),0)$ to be the CDGA quasi isomorphism sending $a \mapsto [a]$ and $b \mapsto [b]$, then the induced map in cohomology  is the identity.
\end{proof}

By using the previous Lemma, we can prove the following 
 \begin{corollary} \label{corollary:cor1}
If the diffeomorphism $\psi$  in Theorem \ref{lemma:2} is the identity map, then $(M_f, I_\pm)$ is biholomorphic to
 $(N \times \mathbb{T}^3_\rho, J \oplus J_\pm)$ with $J$ being the complex structure on the K\"ahler manifold $N$ and $J_\pm$  being the the complex structures on $\mathbb{T}^3_\rho$ described in Lemma \ref{lemma:splitmt}.Therefore $M_f$ is formal and 
\[
 H^k (M_f)= \bigoplus_{\substack{i+j=k}} H^i (N)\otimes H^j (\mathbb{T}^3_ \rho), \quad 
 H^{p,q}_{\Bar{\partial}_\pm}(M_f)= \bigoplus_{\substack{a+c=p,\\ b+d= q}} H^{a,b}_{\Bar{\partial}_i}(N)\otimes H^{c,d}_{\Bar{\partial}_\pm}(\mathbb{T}^3_ \rho).
\]

\end{corollary}
\begin{proof}
Let $\varphi$ be the diffeomorphism
\begin{align*}
    \varphi: M_f &\to N \times \mathbb{T}^3_\rho, \,\,
            [(\underline{x},p,t)] \mapsto (p,[(\underline{x},t)])
\end{align*}
with inverse map 
\begin{align*}
    \varphi^{-1}:  N \times \mathbb{T}^3_\rho &\to M_f,    \, \, 
             (p,[(\underline{x},t)]) \mapsto [(\underline{x},p,t)] .
\end{align*}
Note that both $\varphi$ and $\varphi^{-1}$ are well defined maps since
\[
\begin{split}
    & \varphi([(\underline{x},p,0)]=(p,[(\underline{x},0)])=(p,[(\rho(\underline{x}),1)])=\varphi([(\rho(\underline{x}),p,1)]=\varphi([(f(\underline{x},p),1)].
\end{split}
\]
To prove the statement it suffices to check that $\varphi$ is actually a biholomorphism. \\
In any point $q=[(\underline{x},p,t)]$, $\varphi_*$ identifies the tangent space $T_q M_f= \langle e_{iq}, D_q , \partial \theta_{q}\rangle$ with $T_p N \oplus T_{[(\underline{x},t)]}\mathbb{T}^3_\rho= T_p N \oplus \langle e_{i[(\underline{x},t)]},  \partial \theta_{[(\underline{x},t)]}\rangle $. Observe that $D_q \cong T_p N$. With respect to the decomposition above,
\[
\begin{split}
\varphi_*&=
\begin{pmatrix}
 \mathbf{0}_{4k,3} & I_{4k} & \underline{\mathbf{0}}_{4k,1} \\
 I_{3} & \mathbf{0}_{3,4k} &  \underline{\mathbf{0}}_{3,1}  \\
 \underline{\mathbf{0}}_{1,3} & \underline{\mathbf{0}}_{1,4k} & 1\\
\end{pmatrix}, \ \ \ \ 
\varphi^{-1}_* =
\begin{pmatrix}
 \mathbf{0}_{3,4k} & I_{3} & \underline{\mathbf{0}}_{3,1} \\
 I_{4k} & \mathbf{0}_{4k,3} &  \underline{\mathbf{0}}_{4k,1}  \\
 \underline{\mathbf{0}}_{1,4k} & \underline{\mathbf{0}}_{1,3} & 1\\
\end{pmatrix}, \\
I_\pm &=
\begin{pmatrix}
 \pm M_{3,3} & \mathbf{0}_{3,4k} & \underline{v}_{3,1} \\
 \mathbf{0}_{4k,3} & J &  \underline{\mathbf{0}}_{4k,1}  \\
 \underline{v}_{1,3} & \underline{\mathbf{0}}_{1,4k} & 0\\
\end{pmatrix}, \  \ \ \ 
J \oplus J_\pm=
\begin{pmatrix}
 J & \mathbf{0}_{4k,3} & \underline{\mathbf{0}}_{4k,1}  \\
 \mathbf{0}_{3,4k} & \pm M_{3,3} &  \underline{v}_{3,1}  \\
 \underline{\mathbf{0}}_{1,4k} & \underline{v}_{1,3} & 0\\
\end{pmatrix}
\end{split}
\]
where
\[
\begin{split}
    & M_{3,3}:=
    \begin{pmatrix}
     0 & 0 & 0\\
     0 &0 & -1 \\
     0 & 1 & 0 \\
    \end{pmatrix}, \ \ \ \ 
    \underline{v}_{3,1}:= 
    \begin{pmatrix}
        -1 \\
        0 \\
        0 \\
    \end{pmatrix} \ \ \text{and}\ \ 
    \underline{v}_{1,3}:= 
    \begin{pmatrix}
        1 & 0 & 0
    \end{pmatrix}. 
\end{split}
\]
We compute
 \begin{align*}
&\varphi_* \circ I_\pm = 
\begin{pmatrix}
\mathbf{0}_{4k,3} & J & \underline{\mathbf{0}}_{4k,1}\\
 \pm M_{3,3} & \mathbf{0}_{3,4k} &  \underline{v}_{3,1}  \\
 \underline{v}_{1,3} & \underline{\mathbf{0}}_{1,4k} & 0\\
\end{pmatrix}= (J \oplus J_\pm)\circ \varphi_*, \\
&\varphi^{-1}_* \circ (J \oplus J_\pm) = 
\begin{pmatrix}
\mathbf{0}_{3,4k} & \pm M_{3,3} & \underline{v}_{3,1}\\
 J & \mathbf{0}_{4k,3} &  \underline{\mathbf{0}}_{4k,1} \\
 \underline{\mathbf{0}}_{1,4k} & \underline{v}_{1,3} & 0\\
\end{pmatrix}= I_\pm \circ \varphi^{-1}_*.\\
\end{align*}
Since $M_f= N\times \mathbb{T}^3_\rho$ is the product of two formal manifolds, then $M_f$ is formal itself.  The expression of the De Rham and Dolbeault cohomology of $M_f$  follows by K\"unneth formulas.
\end{proof}
\begin{remark} Observe that the previous Theorem is true also in the hyperK\"ahler case. Indeed, in the first part one can adapt the previous proof setting 
\[
I_\pm =
\begin{pmatrix}
 \pm M_{3,3} & \mathbf{0}_{3,4k} & \underline{v}_{3,1} \\
 \mathbf{0}_{4k,3} & J_i &  \underline{\mathbf{0}}_{4k,1}  \\
 \underline{v}_{1,3} & \underline{\mathbf{0}}_{1,4k} & 0\\
\end{pmatrix}, \  \ \ \ 
J_i \oplus J_\pm=
\begin{pmatrix}
 J_i & \mathbf{0}_{4k,3} & \underline{\mathbf{0}}_{4k,1}  \\
 \mathbf{0}_{3,4k} & \pm M_{3,3} &  \underline{v}_{3,1}  \\
 \underline{\mathbf{0}}_{1,4k} & \underline{v}_{1,3} & 0\\
\end{pmatrix}.
\]
Hence, if in Theorem \ref{lemma:1} we fix $\psi=Id$, then $(M_f, I_\pm)$ is biholomorphic to $(N \times \mathbb{T}^3_\rho, J_i \oplus J_\pm)$, with $i=1,2$. \\
The proof of the formality holds without any change, since it does not depend on the complex structures but only on the K\"ahler property of the manifold. \\
\end{remark}

\section{Dolbeault cohomology} 

In this section we prove  that the generalized Kahler mapping tori  $M_f$,  constructed as in Theorems \ref{lemma:2} and \ref{lemma:1}, is the total space of a holomorphic fibre bundle $p:M_f \to \mathbb{T}^3_{\rho}$ with fibre $N$. To such a fibre bundle is always possible to associate the Borel spectral sequence, which relates the Dolbeault cohomology of the total space $M_f$ with that of the base space $\mathbb{T}^3_\rho$ and of the fibre $N$. \\
This provides a generalization of the trivial case when  the diffeomorphism  of $N$ is the identity map. \\ 
We first recall the following Theorem of A. Borel contained in \cite[Appendix II]{HB}. 
\begin{theorem} \label{theorem:Borel}
Let $p:T \to B$ be a holomorphic fibre bundle, with compact connected fibre $F$ and $T$ and $B$ connected. Assume that $F$ is K\"ahler. Then there exists a spectral sequence $(E_r,d_r)$, with $d_r$ being the restriction of the debar operator $\overline{\partial}$ of $T$ to $E_r$, satisfying the following properties:
\begin{itemize}
    \item $E_r$ is $4$-graded by the fibre degree, the base degree and the type. Let $^{p,q}E_r^{u,v}$ be the subspace of elements of $E_r$ of type $(p,q)$, fibre degree $u$ and base degree $v$. We have that $^{p,q}E_r^{u,v}=0$ if $p+q \neq u+v$ or if one of $p,q,u,v$ is negative. Moreover, $d_r$ maps $^{p,q}E_r^{u,v}$ into $^{p,q+1}E_r^{u+r,v-r+1}$. \\
    \item If $p+q=u+v$
    \[
    ^{p,q}E_2^{u,v}=\sum_{k} H^{k,u-k}_{\overline{\partial}}(B) \otimes H^{p-k,q-u+k}_{\overline{\partial}}(F).
    \]
    \item The Borel spectral sequence converges to $H_{\overline{\partial}}(T)$.
\end{itemize}
\end{theorem} 

From now on we will use the following notation: $M_f^{\pm}=(M_f, I_\pm)$, $\mathbb{T}^{3\pm}_\rho=(\mathbb{T}^3_\rho, J_\pm)$, $N^i=(N,J_i)$, with $i=1,2$. Obviously, in the K\"ahler case $J_1=J_2=J$.\\
Let us define $p: M_f^{\pm} \to \mathbb{T}^{3\pm}_\rho$ to be the holomorphic surjection map, sending $[(\underline{x},q,t)]$ to $[(\underline{x},t)]$. \\
In any point $\xi=[(\underline{x},q,t)]$, a basis of the tangent space $T_\xi M_f$ is provided by $ \langle e_{i\xi}, D_\xi , \partial \theta_{\xi}\rangle$, while a basis of $T_{p(\xi)}\mathbb{T}^{3}$ is provided by $\langle e_{ip(\xi)}, \partial \theta_{p(\xi)}\rangle$. 
With respect to the decomposition above, 
\[
 p_*=\begin{pmatrix}I_{3,3} & \mathbf{0}_{3,4k} & \underline{\mathbf{0}}_{3,1} \\
 \underline{\mathbf{0}}_{1,3}  &
 \underline{\mathbf{0}}_{1,4k} &  1\\
\end{pmatrix}.
\]
One can easily check that $[p_* ,I_\pm] = 0$, where the corresponding matrices of $I_\pm$ and $J_\pm$ are those of Corollary \ref{corollary:cor1}. \\
We have to prove that for each class $[(\underline{x},q)]$ there exists a neighbourhood $U$ and a local holomorphic trivialization $\phi_U: p^{-1} (U)^{\pm} \to U^{\pm} \times N^i$.\\
Let us start with points of the kind $[(\underline{x},t)]$ with $t \neq 0,1$. Then a neighbourhood of these points is given by $V \times I$, where $V$ is an open neighbourhood of $\underline{x}$ in $\mathbb{T}^3$ and $I$ is an open neighbourhood of $t$ in $[0,1]$ not containing $0$ and $1$. \\
Clearly, $p^{-1}U=\{ [(\underline{x},q,t)]\ | \ q \in N \ \ \text{and} \ \ [(\underline{x},t)] \in U \}$ and the local trivialization is provided by 
\[
 \begin{split}
      \phi_U: p^{-1}(U)^{\pm} & \to U^{\pm} \times N^i\\
              [(\underline{x},q,t)] &\mapsto ([(\underline{x},t)],q).
 \end{split}
\]
Let us now consider points of the kind $[\underline{x},0]$. A neighbourhood of such points is given by $U=\pi_{\mathbb{T}^3_\rho}(V \times [0,\varepsilon) \sqcup \rho(V) \times (1-\varepsilon, 1])$ where $\pi_{\mathbb{T}^3_\rho}: \mathbb{T}^3 \times [0,1] \to \mathbb{T}^3_\rho$ and $V$ is an open neighbourhoud of $\underline{x}$. Then $p^{-1}(U)=\pi (V \times N \times [0,\varepsilon) \sqcup \rho(V) \times N \times (1-\varepsilon,1])$. We define a local trivialization on the representatives to be 
\[
 \begin{split}
      \phi_U: p^{-1}(U)^{\pm} & \to U^{\pm} \times N^i\\
              [(\underline{x},q,t)] &\mapsto 
    \begin{cases}
      ([\underline{x},t],q) & \text{if} \ \  t \in [0,\varepsilon) \\
     ([\underline{x},t],\psi^{-1} q)    &  \text{if} \ \  t \in (1-\varepsilon,1]
    \end{cases}.
 \end{split}
\]
Although the $\phi_U$ is defined on the representatives, it does not depend actually on the representative chosen. Indeed, we have to check that $[(\underline{x},q,0)]$ has the same image under $\phi_U$ of $[(\rho(\underline{x}),\psi(q),1)]$. By definition of $\phi_U$, 
\[
 \phi_U[(\rho(\underline{x}),\psi(q),1)]=([\rho(\underline{x}),1], \psi^{-1}\psi (q))=([\underline{x},0],q)=\phi_U([(\underline{x},q,0)]).
\]
The inverse map $\phi_U^{-1}$ is provided by
\[
 \begin{split}
      \phi_U^{-1}: U^{\pm} \times N^i& \to p^{-1}(U)^{\pm} \\
              ([(\underline{x},t)],q) &\mapsto 
    \begin{cases}
      ([\underline{x},q, t]) & \text{if} \ \  t \in [0,\varepsilon) \\
     [(\underline{x},\psi (q), t)]   &  \text{if} \ \  t \in (1-\varepsilon,1]
    \end{cases}.
 \end{split}
\]
Also in this case $\phi_U^{-1}$ is well defined, since
\[
 \phi^{-1}_U([(\rho(\underline{x}),1)], q)=[(\rho(\underline{x}),\psi (q),1)]=[(\underline{x},q, 0)]=\phi_U^{-1}([(\underline{x},0)],q).
\]
As $\phi_U$ and $\phi_U^{-1}$ do not depend on the representatives, without loss of generality  we may choose one of them.  Hence, let us fix the classes $[(\underline{x},q,t)]$ and $[(\underline{x},t)]$ represented by $(\underline{x},q,t)$ and $(\underline{x},t)$ with $t \in [0,\varepsilon)$, respectively. Then, by definition of both $\phi_U$ and $\phi_U^{-1}$ we have that 
\[ 
   \phi_U ([(\underline{x},q,t)])=([(\underline{x},t)],q) \ \ \text{and} \ \ \phi_U ([(\underline{x},t)],q)=[(\underline{x},q, t)],
\]
the holomorphy of both maps easily follows. \\
Since the fibre $N$ of the bundle $p:M_f^{\pm}\to \mathbb{T}^3_\rho$  is clearly compact and (hyper)K\"ahler, we have two associated Borel spectral sequences described in Theorem \ref{theorem:Borel}. In particular we have the following result 
\begin{theorem} \label{theorem:holfib}
If $M_f$ is the mapping torus constructed as in the Theorems \ref{lemma:2} and \ref{lemma:1}, then $M_f$ is the total space of the holomorphic fibre bundle
\[
   N \to M_f \to \mathbb{T}^3_\rho.
\]
Then we have two associated Borel spectral sequences $(E^{\pm}_r,d^{\pm}_r)$ satisfy the following properties:
\begin{itemize}
\item If $p+q=u+v$
    \[
    ^{p,q}E_2^{u,v  \pm }=\sum_{k} H^{k,u-k}_{\overline{\partial}_\pm}(\mathbb{T}^3_\rho) \otimes H^{p-k,q-u+k}_{\overline{\partial}_i}(N).
    \] 
\item The Borel spectral sequences converge respectively to $H_{\overline{\partial}_\pm}(M_f)$. \end{itemize} 
\end{theorem}
Note that when $\psi$ is the identity map, then we have already observed in  Corollary \ref{corollary:cor1} and in the following remark
that $M_f^{\pm}$ is biholomorphic to $N^i \times \mathbb{T}_{\rho}^{3\pm}$. In terms of our fibre bundle it is equivalent to say that the bundle is trivial. Since $d_r^{\pm}$ are the restriction of the debar operators $\overline{\partial}_\pm$ of $M_f$, we have that if $p+q=u+v$ then $d_2^{\pm}(^{p,q}E_2^{u,v \pm})= \overline{\partial}_\pm (^{p,q}E_2^{u,v \pm})=0$, as each element in the $E_2^\pm$ term is already a global closed form on $M_f$. It follows that the sequences both degenerate at $r=2$. This result is perfectly consistent with that described in Corollary \ref{corollary:cor1}, where we obtained the same result applying the K\"unneth formula.

\section{Explicit examples}

We will give now some explicit examples obtained applying the construction described in the previous sections. 

\begin{example} \label{example:identity} Let $\rho$ be the $\mathbb{T}^3$ diffeomorphism defined by \eqref{eqn:rho} and let $\{e_i\}$ and $\{e^i\}$ be the basis described in Example \ref{example:basis}. We have already observed in the description of Example \ref{example:basis} that $\rho$ preserves the basis $e^i(t)$ and one can easily check that $\big(\frac{1}{l(t)} \big)' \cdot \frac{1}{a_1}= (e^t)' \cdot e^{-t}=1 \neq 0$. 
We apply the Theorem \ref{lemma:1} in the case of $N$ being $4$-torus $\mathbb{T}^4$ endowed with the  standard flat hyperK\"ahler structure $(J_1,J_2,J_3, \omega_1, \omega_2, \omega_3)$, defined as follows
\[
\begin{split}
&J_1\bigg(\frac{\partial}{\partial x^4}\bigg)=\frac{\partial}{\partial x^5}, \quad J_1\bigg(\frac{\partial}{\partial x^6}\bigg)=-\frac{\partial}{\partial x^7},\\
&J_2\bigg(\frac{\partial}{\partial x^4}\bigg)=-\frac{\partial}{\partial x^7}, \quad J_2\bigg(\frac{\partial}{\partial x^5}\bigg)=\frac{\partial}{\partial x^6}, \\
&J_3=J_1J_2, \ \text{and} \ k=\sum_{i=4}^7 (dx^i)^2,
\end{split}
\]
where $(x^4,x^5,x^6,x^7)$ are the standard coordinates on $\mathbb{T}^4 \cong \mathbb{Z}^4 \backslash \mathbb{R}^4$ and $\psi$ being the identity map.
Note that  
\[
   M_f= \frac{\mathbb{T}^3 \times \mathbb{T}^4\times [0,t_0]}{(q,0) \sim (f(q),t_0) },
\]
where $f(\underline{x},\underline{y})=(\rho(\underline{x}),\underline{y})$. The only difference with the general case is that now the Cartesian product is done with $[0,t_0]$. \\
By Theorem \ref{lemma:1}, $M_f$ can be endowed with a non-split generalized Kahler structure $(g,I_\pm)$. Moreover, we may apply Corollary \ref{corollary:cor1}, to conclude that 
$(M_f, I_\pm)$ is a formal manifold biholomorphic to $(\mathbb{T}^4 \times \mathbb{T}^3_ \rho, J_i \oplus J_\pm)$. Furthermore, $M_f$ does not admit any K\"ahler metric and does not satisfy the $dd^c$-lemma.
\\
We can describe  $M_f$  also as a compact solvmanifold. Let us consider the $8$-dimensional unimodular almost abelian Lie group $G_{8}^{p,q}$, $q \in \mathbb{R}\setminus \{0\} $, with structure equations
\begin{equation} \label{eqn:maurer}
    \begin{split}
        & df^1=f^1 \wedge f^{8} \ , \  df^2= -\frac{1}{2} f^2 \wedge f^{8} + p f^3 \wedge f^{8}\ , \
        df^3=-p f^2\wedge f^{8} -\frac{1}{2}f^3 \wedge f^{8}, \\
        & df^{4}= q f^{5} \wedge f^{8},  df^{5}=-q f^{4} \wedge f^{8},\\
        & df^{6}= q f^{7} \wedge f^{8},  df^{7}=-q f^{6} \wedge f^{8},\\
        & df^{2n}=0.
    \end{split}
\end{equation}
$G_{p,q}^{8}$ is the semidirect product $\mathbb{R}^7 \rtimes_{\widetilde{\varphi}}\mathbb{R}$, where
\[
\begin{split}
    \widetilde{\varphi}(t)&=\begin{pmatrix}
              \varphi(t) &  0 & 0\\
              0 & R_q(t) & 0 \\
              0 & 0 & R_q(t) \\
              \end{pmatrix},
\end{split}
\]
with $\varphi$ defined as in the Remark \ref{remark:solvmanifold} and $R_q(t)$ being the rotation matrix
\[ 
R_q(t)= \begin{pmatrix}
  \cos(qt) & -\sin(qt) \\
  \sin(qt) & \cos(qt) \\   
\end{pmatrix}.
\]
We have already observed that for $t=t_0$, $\varphi=\rho(t_0)$ is similar to an integer matrix $A$. Hence, fixing $q=\frac{2\pi}{t_0}$, $\widetilde{  \varphi}$ is similar to the integer matrix $\operatorname{diag}(A, I_2, I_2)$ via an invertible matrix $\widetilde{P}$. The mapping torus constructed above is the quotient $\Gamma \backslash G_{p,q}^8$, where $\Gamma= \widetilde{P} \mathbb{Z}^7 \rtimes t_0 \mathbb{Z}$.
\end{example}

 Now we give an example which is not biholomorphic to the product of $N$ and ${\mathbb T}^3_{\rho}$.
\smallskip

\begin{example}
Let $\rho$, $\{e^i\}$ and $(\mathbb{T}^4, k,J_i)$ be defined as in the previous Example and let $\psi$ be the $\mathbb{R}^4$-rotation $(x^4,x^5,x^6,x^7)\mapsto (x^5,-x^4,x^7,-x^6)$.
Since $\psi$ is represented by an integer matrix, $\psi$ descend to a diffeomorphism of the flat torus $\mathbb{T}^4$. \\
We prove that $\psi$ is holomorphic with respect to $J_i$ and preserves the hyperkahler structure $(\mathbb{T}^4,k,J_i)$. Indeed, $[\psi_*, J_1] = [\psi_*, J_2]=0$
and
\[
\begin{split}
& \psi^* k= \psi^*\big(\sum_{i=4}^7 (dx^i)^2 \big)=(dx^5)^2+ (dx^4)^2 + (dx^7)^2 +(dx^6)^2=k, \\
& \psi^* (\omega_1)= \psi^* (dx^4 \wedge dx^5 - dx^6 \wedge dx^7)=\omega_1, \\
& \psi^* (\omega_2)= \psi^* (-dx^4 \wedge dx^7 + dx^5 \wedge dx^6)=\omega_2, \\
& \psi^* (\omega_3)= \psi^* (-dx^4 \wedge dx^6 - dx^5 \wedge dx^7)=\omega_3.
\end{split}
\]
Then
\[
   M_f= \frac{\mathbb{T}^3 \times \mathbb{T}^4\times [0,t_0]}{(q,0) \sim (f(q),t_0) },
\]
where $f(\underline{x},\underline{y})=(\rho(\underline{x}),\psi(\underline{y}))$, is a non-split generalized K\"ahler manifold by Theorem \ref{lemma:1}, which does not admit any K\"ahler metric and does not satisfy the $dd^c$-lemma.\\ 
Observe that the mapping torus $M_f$ can be constructed also as a  compact quotient of $G_{p,q}^{8}$, for $q=\frac{\pi}{2t_0}$ by a lattice $\Gamma$. More precisely, $G_{p,q}^{8}= \Gamma \backslash G_{p,q}^{8}$, where $\Gamma=\widetilde{P}\mathbb{Z}^7 \rtimes t_0 \mathbb{Z}$ and $\widetilde{P}$ being the matrix of change of basis between $\widetilde{\varphi}_{\frac{\pi}{2t_0}}(t_0)$ and $\operatorname{diag}(A, \Lambda, \Lambda)$ with
\[
\Lambda= \begin{pmatrix}
     0 & 1 \\
     -1 & 0 \\
\end{pmatrix}.
\]
Now we compute the De Rham cohomology of  $M_f$, using the fact that $M_f$ is the mapping torus  of $\mathbb{T}^3 \times \mathbb{T}^4 \cong \mathbb{T}^7$  by  the  diffeomorphism $f=(\rho,\psi)$ described above. \\

By \cite[Lemma 12]{BFM} the cohomology of degree $r$ is given by the isomorphism 
\[
   H^r(M_\varphi)=K^r \oplus C^{r-1},
\]
where $K^r=\ker (\varphi_r^*- Id)$ and $C^{r}=\operatorname{coker}(\varphi_r^*- Id)$. \\
Hence we need to fix a basis of the $r$-degree cohomology of $\mathbb{T}^3 \times \mathbb{T}^4\cong \mathbb{T}^7$, for each $r=1,\dots, 8$.\\
Recall that $H^{r}(\mathbb{T}^3 \times \mathbb{T}^4)=\langle \  [dx^{I} \wedge dx^{J}]\ | \ \abs{I} + \abs {J}=r \ \rangle,$ where $I$ is a multi-index of length $\abs{I}$  with indexes in $\{1,\dots,3\}$ and  $J$ is a multi-index of length $\abs{J}$  with indexes in $\{4,\dots,7\}$. \\
Note that on $dx^{I} \wedge dx^{J}$, $f=(\rho, \psi)$ acts in the following way:
\[
    f^*(dx^{I} \wedge dx^{J})=\rho^* (dx^{I}) \wedge \psi^*(dx^{J}).
\]
For any $(i,j)$ with $i+j \le 8$ consider the pull-back 
\[
      f_{(i,j)}^{*}=(\rho_i^*, \psi_j^*): H^{i}(\mathbb{T}^3) \otimes H^{j}(\mathbb{T}^4) \to H^{i}(\mathbb{T}^3) \otimes H^{j}(\mathbb{T}^4).
\]
Then, we can show that  for each $[\alpha \wedge \beta] \in H^{i}(\mathbb{T}^3) \otimes H^{j}(\mathbb{T}^4)$, \ $f_{(i,j)}^{*}[\alpha \wedge \beta]=[\alpha \wedge \beta]$ if and only if $\rho^*_{i}[\alpha]=[\alpha] $ and $\psi^*_{j}[\beta]=[\beta]$.
Indeed,  by contradiction, let us assume that for each $[\alpha \wedge \beta] \in H^{i}(\mathbb{T}^3) \otimes H^{j}(\mathbb{T}^4)$, \ 
\[f_{(i,j)}^{*}[\alpha \wedge \beta]=[\alpha \wedge \beta] \ \  \wedge \ \ \rho^*_{i}[\alpha]\neq[\alpha].\] 
Let $[\beta]$ be a non-zero class in $H^0(\mathbb{T}^4)$. Then clearly $\psi_0^*[\beta]=[\beta]$. By hypothesis 
\[
   \beta \cdot f_{(i,0)}^{*}[\alpha]= \beta \cdot \rho_i^{*}[\alpha]= \beta \cdot [\alpha].
\]
Since $\beta$ is non-zero, this provides a contradiction. The converse is obvious. 

In other words we have that if in degree $(i,j)$ either  $\rho^*_i(dx^{I}) \neq dx^{I}$ for each $\abs{I}=i$ or $\psi^*_j(dx^{J}) \neq dx^{J}$ for each $\abs{J}=j$, then $f^*_{(i,j)}-Id$ has trivial kernel and cokernel. \\
In the following we will denote by $K^r_{\mathbb{T}^3}$ and  $C^r_{\mathbb{T}^3}$ the $\ker$ and the $\operatorname{coker}$ of $\rho^*_r-Id$, respectively and by $K^r_{\mathbb{T}^4}$ and  $C^r_{\mathbb{T}^4}$ the $\ker$ and the $\operatorname{coker}$ of $\psi^*_r-Id$, respectively. By the latter remark, we have that $K^{(i,j)}=K^i_{\mathbb{T}^3} \wedge K^j_{\mathbb{T}^4}$.\\
Since $\mathbb{T}^3_\rho$ is an Inoue surface in the family $S_M$,  we get 
\[
\begin{split}
    & K^0_{\mathbb{T}^3}= C^0_{\mathbb{T}^3}= \langle 1 \rangle,  \ \ \ 
    K^1_{\mathbb{T}^3}= C^1_{\mathbb{T}^3}= 0, \\
    & K^2_{\mathbb{T}^3}= C^2_{\mathbb{T}^3}= 0, \ \ \ \ 
     K^3_{\mathbb{T}^3}= C^3_{\mathbb{T}^3}= \langle [e^{123}] \rangle, \\
\end{split}
\]
Moreover, if we compute in each degree $K^r_{\mathbb{T}^4}=\ker(\psi_r^*-Id)$ and $C^r_{\mathbb{T}^4}=\operatorname{coker}(\psi_r^*-Id)$ with respect to the basis $\{[dx^J]\}$ we obtain 
\[
\begin{split}
     &K^0_{\mathbb{T}^4}= C^0_{\mathbb{T}^4}= \langle 1 \rangle, \ \ \ K^1_{\mathbb{T}^4}= C^1_{\mathbb{T}^4}= 0,  \\
     &K^2_{\mathbb{T}^4}= C^2_{\mathbb{T}^4}= \langle \ [dx^{45}], [dx^{46}+dx^{57}], [dx^{47}-dx^{56}], [dx^{67}] \ \rangle,\\
     &K^3_{\mathbb{T}^4}= C^3_{\mathbb{T}^4}= 0, \ \ \  K^4_{\mathbb{T}^4}= C^4_{\mathbb{T}^4}= \langle \  [dx^{4567}] \ \rangle. \\
\end{split}
\]
We may now compute in each degree $r=0,\dots,8$ the spaces $K^r$ and $C^r$, which are provided by
\[
\begin{split}
&K^0=C^0=\langle 1 \rangle, \ \ \ K^1=C^1=0, \\ &K^2=K^2_{\mathbb{T}^4}=C^2_{\mathbb{T}^4}=C^2=\langle \ [dx^{45}], [dx^{46}+dx^{57}], [dx^{47}-dx^{56}], [dx^{67}] \ \rangle, \\
&K^3=K^3_{\mathbb{T}^3}=C^3_{\mathbb{T}^3}=C^3=\langle [e^{123}] \rangle, \\
& K^4= K^4_{\mathbb{T}^4}=C^4_{\mathbb{T}^4}=C^4=\langle [dx^{4567}] \rangle, \\
& K^5=K^3_{\mathbb{T}^3} \wedge K^2_{\mathbb{T}^4}= C^5= \langle [e^{123}\wedge dx^{45}], [e^{123} \wedge (dx^{46}+ dx^{57})],\\
&[e^{123} \wedge (dx^{47}- dx^{56})], [e^{123} \wedge dx^{67}]\rangle, \\
& K^6=C^6=0,  \ \ \ K^7=K^3_{\mathbb{T}^3}\wedge K^4_{\mathbb{T}^4}=C^7=\langle [e^{123}\wedge dx^{4567}] \rangle.
\end{split}
\]
The computation of $H^*(M_f)$ is now trivial,   
\begin{align*}
    H^1 (M) &=  \langle [\theta]   \rangle, \\
    H^2 (M) &=  \langle  [dx^{45}],[dx^{46}+dx^{57}],[dx^{47}-dx^{56}] ,[dx^{67}]  \rangle,\\
    H^3 (M) &=  \langle  [e^{123}], [\theta \wedge dx^{45}], [\theta \wedge (dx^{46}+dx^{57})],[\theta \wedge (dx^{47}-dx^{56})] ,[\theta \wedge dx^{67}] \rangle , \\
    H^4 (M) &=  \langle  [dx^{4567}],[\theta \wedge e^{123}]\rangle,\\
    H^5 (M) &=  \langle  [e^{123}\wedge dx^{45}], [e^{123} \wedge (dx^{46}+ dx^{57})],[e^{123} \wedge (dx^{47}-dx^{56})], [e^{123} \wedge dx^{67}],[\theta \wedge dx^{4567}]\rangle,\\
    H^6 (M) &=  \langle [\theta \wedge e^{123}\wedge dx^{45}], [\theta \wedge e^{123} \wedge (dx^{46}+ dx^{57})],[\theta \wedge e^{123} \wedge ( dx^{47}-dx^{56})], [\theta \wedge e^{123} \wedge dx^{67}] \rangle,\\
    H^7 (M) &=  \langle [e^{123}\wedge dx^{4567}] \rangle,\\
    H^8 (M) &=  \langle [\theta \wedge e^{123}\wedge dx^{4567}] \rangle.\\
\end{align*}

Let $I$ be the set of indexes $\{1,2,3,4\}$. The minimal model of $M_f$ is provided by $(\bigwedge V,d,\varphi)$, where
\[
  \bigwedge V= \bigwedge (a) \otimes \bigwedge_{i\in I} (b_i) \otimes \bigwedge_{(i,j)\in I \times I-(1,4)}(c,\lambda_{ij}),  
\]
with degrees $\abs{a}=1$, \ $\abs{b_i}=2$, \ $\abs{c}=\abs{\lambda_{ij}}=3$ and differential
\[
da=db_i=dc=0, \ \ d\lambda_{ij}=b_i\wedge b_j.
\]
Note that all the wedge product $b_i \wedge b_j$ are exact, except for $b_1 \wedge b_4$. \\
The quasi isomorphism $\varphi: (\bigwedge V, d) \to (\Omega^*(M_f),d) $ is defined as follows:
\[
\begin{split}
&\varphi(a)=\theta, \ \ \varphi(b_1)=dx^{45}, \ \ \varphi(b_2)=dx^{46}+dx^{57}, \\
&\varphi(b_3)=dx^{47}-dx^{56}, \ \ \varphi(b_4)=dx^{67},\\
&\varphi(c)=e^{123}, \ \ 
\varphi(\lambda_{ij})=0.
\end{split}
\]
Observe that under the identification provided by $\varphi$, the cohomology of $(\bigwedge V,d)$ is precisely that of $M_f$.\\
We claim $M_f$ to be formal. Indeed, if we define a quasi isomorphism $\nu:(\bigwedge V,d) \to \left(H^*(\bigwedge V,d),0\right)$ on the generators of $\bigwedge V$ to be 
\begin{align*}
       &\nu(a)=[a], \ \ 
       \nu(b_i) =[b_i], \\
       &\nu(c)= [c], \ \ 
       \nu(\lambda_{ij})= 0,\\
\end{align*}
then the induced map in cohomology is the identity. The formality follows.\\
Now we compute the Dolbeault cohomologies of $M_f$, by using the Borel spectral sequences described in Theorem \ref{theorem:holfib}.  
As $\mathbb{T}^3_\rho$ is an Inoue surface in the family $S_M$, $\mathbb{T}^3_\rho$ is endowed with the following basis of $(1,0)$ forms
\begin{align*}
    & \varphi^1_+=e^2+ ie^3, \ \ \ \ \varphi^1_-=e^3+ie^2, \\
    & \varphi^2_\pm=e^1 + i \theta, \\
\end{align*}
whose differentials satisfy
\begin{align*}
    & d\varphi^1_\pm= \frac{\alpha - i \beta_\pm}{2i} \  \varphi^{12} - \frac{\alpha - i \beta_\pm}{2i} \ \varphi^{1\bar{2}}, \\
    & d\varphi^2_\pm=- i \alpha \ \varphi^{2\bar{2}},
\end{align*}
where $\alpha=-\frac{1}{2}$ and $\beta_\pm=\pm \ p$.\\
Then:
\begin{equation} 
H^{\bullet, \bullet}_{\bar{\partial}_\pm} (\mathbb{T}^3_\rho) = \mathbb{C}\langle 1 \rangle  \oplus \mathbb{C}\langle [\varphi_\pm^{\bar{2}}] \rangle \oplus \mathbb{C}\langle [\varphi_\pm^{12\bar{1}}] \rangle \oplus \mathbb{C}\langle [\varphi_\pm^{12\bar{1}\bar{2}}] \rangle.
\end{equation}
Let us fix the 
invariant coframe $\{\eta^1_{i},\eta^2_{i}\}$ of $T^{1,0} _i\mathbb{T}^4$ provided by
\begin{align*}
    &\eta^1_{1}=dx^4+ i dx^5, \ \ \  \eta^2_{1}=dx^7+idx^6, \\
    &\eta^1_{2}=dx^5+ i dx^6, \ \ \ \eta^2_{2}=dx^7+idx^4, \\
\end{align*}
with structure equations $d\eta^1_i=0$ and $d\eta^2_i=0$. Then:
\begin{align*}
   H^{\bullet,\bullet}_i(\mathbb{T}^4)&= \mathbb{C}\langle 1 \rangle  \oplus \mathbb{C}\langle [\eta_i^{k}] \rangle
   \oplus \mathbb{C}\langle [\eta_i^{\bar{k}}]\rangle 
   \oplus \mathbb{C}\langle [\eta_i^{12}] \rangle \oplus \mathbb{C}\langle [\eta_i^{\overline{12}}] \rangle \\
   &
   \oplus \mathbb{C}\langle [\eta_i^{k \overline{h}}] \rangle \oplus \mathbb{C} \langle [\eta_i^{12\overline{k}}] \rangle \oplus \mathbb{C}\langle [\eta_i^{k \overline{12}}] \rangle \oplus \mathbb{C}\langle [\eta_i^{12 \overline{12}}] \rangle, \\
\end{align*}
where $k=1,2$.\\
By Theorem \ref{theorem:holfib}, for $p+q=u+v$ 
    \[
    ^{p,q}E_2^{u,v  \pm }=\sum_{k} H^{k,u-k}_{\overline{\partial}_\pm}(\mathbb{T}^3_\rho) \otimes H^{p-k,q-u+k}_{\overline{\partial}_i}(\mathbb{T}^4),
    \] 
with $d_2:\ ^{p,q}E_2^{u,v}\mapsto \ ^{p,q+1}E_2^{u+2,v-1}$. We claim that the two Borel spectral sequences degenerate at $r=2$. The proof that $d_2$ is zero is trivial except for $(p,q,u,v)=(2,1,1,2)$ and $(p,q,u,v)=(2,2,1,3)$. Let us investigate the first case. Let $(p,q,u,v)=(2,1,1,2)$. The differential $d_2$ maps $^{2,1}E_2^{1,2}=H^{0,1}_\pm (\mathbb{T}^3_\rho)\otimes H^{2,0}_i(\mathbb{T}^4)=\langle [\varphi_\pm^{\overline{2}}] \otimes  [\eta_i^{12}] \rangle $ in $^{2,2}E_2^{3,1}=H^{2,1}_\pm (\mathbb{T}^3_\rho)\otimes H^{0,1}_i(\mathbb{T}^4)=\langle [\varphi_\pm^{12\overline{1}}] \otimes  [\eta_i^{\overline{k}}]\rangle$.\\
Recall that $d_2$ is precisely the debar operator $\overline{\partial}_\pm$ of $M_f$. Hence, since the fibration $p:M_f \to \mathbb{T}^3_\rho$ is holomorphic, $d_2$ acts on $^{p,q}E_2^{u,v}$ as $id \otimes d_2$. We get
\[
  d_2(^{2,1}E_2^{1,2})=H^{0,1}_\pm (\mathbb{T}^3_\rho)\otimes d_2(H^{2,0}_i(\mathbb{T}^4))=H^{0,1}_\pm (\mathbb{T}^3_\rho)\otimes d_2(^{2,0}E_{2}^{0,2\pm})=0.
\]
The proof of the second case proceeds in the same way. 
\end{example}
As already observed, the previous examples are both diffeomorphic to solvmanifolds. A class of examples not diffeomorphic to solvmanifolds can be constructed  by taking  a $K3$  surface 
as the (hyper)K\"ahler manifold $N$. Indeed, in this latter case by Theorem \ref{theorem:holfib} we have that the holomorphic fibration
\[
   N \to M_f \to \mathbb{T}^3_\rho
\]
induces the long exact sequence of homotopy groups
\[
 \dots \to \pi_{k+1}(\mathbb{T}^3_\rho) \to \pi_{k} (N) \to \pi_k(M_f) \to \pi_{k}(\mathbb{T}^3_\rho) \to \dots. 
\]
If we focus on $k=1$ we get
\begin{equation} \label{eqn:homotopy}
   \dots \to \pi_{2}(\mathbb{T}^3_\rho) \to \pi_{1} (N) \to \pi_1(M_f) \to \pi_{1}(\mathbb{T}^3_\rho) \to \pi_{0}(N) \dots .
\end{equation}
Observe that \eqref{eqn:homotopy} reduces to the short exact sequence
\[
  0 \to \pi_{1} (N) \to \pi_1(M_f) \to \pi_{1}(\mathbb{T}^3_\rho) \to 0
\]
as the Inoue surfaces in the family $S_M$ are solvmanifolds (and hence in particular they are aspherical) and $N$ is a path connected manifold. Since by assumption $N$ is a $K3$ surface, $\pi_{1} (N)$ is trivial and 
\[
   \pi_1(M_f) \cong \pi_{1}(\mathbb{T}^3_\rho).
\]
By contradiction, let us assume that $M_f$ is a solvmanifold. Since $\mathbb{T}^3_\rho$ is another solvmanifold with isomorphic fundamental group, $M_f$ is diffeomorphic to $\mathbb{T}^3_\rho$. Clearly, this is an absurd as $\dim(\mathbb{T}^3_\rho)=4<\dim(M_f)$.

\smallskip

Diffeomorphisms of  $K3$ surfaces preserving  K\"ahler structures have been described for instance in \cite[Theorem 1.8]{FL}.

\smallskip
\textbf{Acknowledgements}. 
Anna Fino is partially supported by Project PRIN 2017 “Real and complex manifolds: Topology, Geometry and Holomorphic Dynamics”, by GNSAGA (Indam) and by a grant from the Simons Foundation (\#944448). 
We would like to thank Giovanni Bazzoni and Gueo Grantcharov for many useful comments and suggestions.

\end{document}